\documentclass[11pt]{amsart}
\usepackage{amsfonts}
\usepackage{mathrsfs}
\usepackage{amscd}
\usepackage{amsmath}
\usepackage{amssymb}
\usepackage{bbm}
\usepackage{latexsym}
\usepackage{lscape}
\usepackage{xypic}
\usepackage{url}

\newtheorem{thm}{Theorem}[]
\newtheorem{prop}[thm]{Proposition}
\newtheorem{lem}[thm]{Lemma}

\newtheorem{lem-def}[thm]{Lemma-Definition}
\newtheorem{cor}[thm]{Corollary}
\theoremstyle{remark}

\newtheorem{rmk}{Remark}[section]
\theoremstyle{definition}

\numberwithin{equation}{section}

\input xy
\xyoption{all}

\newcommand{\quash}[1]{}  
\newcommand{\nc}{\newcommand}
\nc{\on}{\operatorname}

\newcommand{\frakb}{{\mathfrak b}}
\newcommand{\frakc}{{\mathfrak c}}

\newcommand{\frakg}{{\mathfrak g}}
\newcommand{\frakh}{{\mathfrak h}}

\newcommand{\frakk}{{\mathfrak k}}
\newcommand{\frakl}{{\mathfrak l}}
\newcommand{\frakm}{{\mathfrak m}}
\newcommand{\frakn}{{\mathfrak n}}

\newcommand{\frakp}{{\mathfrak p}}
\newcommand{\frakq}{{\mathfrak q}}

\newcommand{\fraks}{{\mathfrak s}}
\newcommand{\frakt}{{\mathfrak t}}

\newcommand{\frakz}{{\mathfrak z}}

\newcommand{\frakZ}{{\mathfrak Z}}
\newcommand{\bbA}{{\mathbb A}}

\newcommand{\bbC}{{\mathbb C}}

\newcommand{\bbG}{{\mathbb G}}

\newcommand{\bbM}{{\mathbb M}}

\newcommand{\bbP}{{\mathbb P}}
\newcommand{\bbQ}{{\mathbb Q}}
\newcommand{\bbR}{{\mathbb R}}

\newcommand{\bbZ}{{\mathbb Z}}
\newcommand{\calA}{{\mathcal A}}
\newcommand{\calB}{{\mathcal B}}

\newcommand{\calD}{{\mathcal D}}
\newcommand{\calE}{{\mathcal E}}
\newcommand{\calF}{{\mathcal F}}
\newcommand{\calG}{{\mathcal G}}
\newcommand{\calH}{{\mathcal H}}

\newcommand{\calL}{{\mathcal L}}

\newcommand{\calO}{{\mathcal O}}

\newcommand{\calS}{{\mathcal S}}

\newcommand{\calV}{{\mathcal V}}

\nc{\al}{{\alpha}} \nc{\be}{{\beta}}
\newcommand{\ga}{{\gamma}}
\nc{\ve}{{\varepsilon}} \nc{\Ga}{{\Gamma}}
\newcommand{\la}{{\lambda}}
\nc{\La}{{\Lambda}}

\nc{\ad}{{\on{ad\ \!}}}
\newcommand{\Ad}{{\on{Ad}}}
\nc{\Adm}{{\on{Adm}}} \nc{\aff}{{\on{aff}}}
\nc{\Aff}{{\mathbf{Aff}}}
\newcommand{\Aut}{{\on{Aut}}}
\nc{\Bun}{{\on{Bun}}}

\nc{\der}{{\on{der}}}

\nc{\diag}{{\on{diag}}}
\newcommand{\End}{{\on{End}}}
\nc{\Fl}{{\calF\ell}}
\newcommand{\Gal}{{\on{Gal}}}
\newcommand{\Gr}{{\on{Gr}}}
\newcommand{\Hom}{{\on{Hom}}}
\nc{\IC}{{\on{IC}}}

\nc{\Id}{{\on{Id}}}

\nc{\Ind}{{\on{Ind}}}
\newcommand{\Lie}{{\on{Lie\ \!}}}
\newcommand{\Pic}{{\on{Pic}}}

\newcommand{\Res}{{\on{Res}}}
\nc{\res}{{\on{res}}}

\newcommand{\Spec}{{\on{Spec\ \!}}}

\nc{\tr}{{\on{tr}}}

\newcommand{\GL}{{\on{GL}}}
\nc{\GSp}{{\on{GSp}}} \nc{\GU}{{\on{GU}}}

\nc{\SL}{{\on{SL}}} \nc{\SU}{{\on{SU}}} \nc{\SO}{{\on{SO}}}

\nc{\bFl}{{\overline{\Fl}}} \nc{\bU}{{\overline{U}}}
\nc{\wGr}{{\widetilde{\Gr}}} \nc{\wJG}{{\widetilde{\calL^+\calG}}}
\nc{\wLG}{{\widetilde{\calL\calG}}} \nc{\Op}{{\on{Op}}}
\nc{\crit}{{\on{crit}}} \nc{\Vac}{{\on{Vac}}}
\nc{\Ran}{{\on{Ran}}}
\nc{\Rat}{{\on{Rat}}}
\nc{\RanG}{{\Ran_\calG}}
\nc{\fS}{{\on{fS}}}

\nc{\ppars}{(\!(s)\!)}

\nc{\opp}{\on{opp}}

\def\xcoch{\mathbb{X}_\bullet}

\title{Frenkel-Gross' irregular connection and Heinloth-Ng\^{o}-Yun's are the same}
\author{Xinwen Zhu}
\address{Xinwen Zhu: Department of Mathematics, California Institute of Technology, Pasadena, CA, 91125, USA.}
\email{xzhu@caltech.edu}
\begin{document}
\maketitle

\begin{abstract}
We show that the irregular connection on $\bbG_m$ constructed by
Frenkel-Gross (\cite{FG}) and the one constructed by
Heinloth-Ng\^{o}-Yun (\cite{HNY}) are the same, which confirms
Conjecture 2.16 of \cite{HNY}.
\end{abstract}

\bigskip

\bigskip
 
\centerline{\sc Contents}
\medskip

\noindent  \ \ \  \ \ \ Introduction\\
\S~1.   \ Kac-Moody chiral algebras\\
\S~2.   \ Hecke eigensheaves\\
\S~3.  \ Some geometry of the local Hitchin map\\
\S~4.  \ Endomorphisms of some vacuum modules\\
\S~5.  \ Proof of a conjecture in \cite{HNY}\\
 \phantom{aa} \  \ References\\

\bigskip
\medskip

 \medskip

\section*{Introduction}
We show that the irregular connection on $\bbG_m$ constructed by
Frenkel-Gross (\cite{FG}) and the one constructed by
Heinloth-Ng\^{o}-Yun (\cite{HNY}) are the same, which confirms
Conjecture 2.16 of \cite{HNY}.

The proof is simple, modulo the big machinery of quantization of
Hitchin's integrable systems as developed by Beilinson-Drinfeld
(\cite{BD}). The key idea is as follows: instead of directly showing that the two connections constructed in \cite{FG} and \cite{HNY} are isomorphic, we go the opposite direction by showing that their corresponding automorphic sheaves under the Langlands correspondence are isomorphic. More precisely,
Let $\calE$ be the irregular
connection on $\bbG_m$ as constructed by Frenkel-Gross. It admits a
natural oper form. We apply (a variant of) Beilinson-Drinfeld's machinery to
produce an automorphic D-module on the corresponding moduli space of
$G$-bundles, with $\calE$ its Hecke eigenvalue. We show that this
automorphic D-module is equivariant with respect to the unipotent
group $I(1)/I(2)$ (see \cite{HNY} or below for the notation) against a
non-degenerate additive character $\Psi$. By the uniqueness of such
D-modules on the moduli space, the automorphic
D-module constructed using Beilinson-Drinfeld's machinery is the
same as the automorphic D-module explicitly constructed by
Heinloth-Ng\^{o}-Yun. Since the irregular connection on $\bbG_m$
constructed in \cite{HNY} is by definition the Hecke-eigenvalue of
this automorphic D-module, it must be the same as $\calE$. We emphasize here that the simple nature of the proof is due to the existence of the ``Galois-to-automorphic" direction for the Langlands 
correspondence over $\bbC$.  Similarly, the existence of such direction over finite fields (which remains conjectural) would imply the unicity of the Kloosterman sheaves as conjectured in \cite{HNY}.

Let us brief describe the main new ingredients needed in the proof of the conjecture. Let $X$ be a smooth projective curve over $\bbC$ and let $\calG$ be a smooth affine group scheme over $X$. Let $\Bun_\calG$ denote the moduli stack of $\calG$-torsors on $X$.
Assume that there is some dense open subset $U\subset X$ such that $\calG|_U=G\times U$ for some reductive group $G$ over $\bbC$. Then there is the usual Hitchin map
\[h^{cl}: T^*\Bun_\calG\to \on{Hitch}(U),\]
where $\on{Hitch}(U)$ is the Hitchin base for $U$ (which is not finite dimensional if $U$ is affine, see \eqref{Hitchin base}). The map factors through a finite dimensional closed subscheme $\on{Hitch}(X)_\calG\subset \on{Hitch}(U)$.

The new difficulty is that, unlike the case when $\calG=G\times X$ is constant, in general $h^{cl}$ is not a completely integrable system, and in particular $\bbC[\on{Hitch}(X)_{\calG}]$ does not provide a maximal Poisson commuting subalgebra of the ring of regular functions on $T^*\Bun_\calG$. For every $x\in X\setminus U$, let $K_x:=\calG(\calO_x)$ and let $P_x$ be the normalizer of $K_x$ in $G(F_x)$, where $\calO_x$ denotes the complete local ring at $x$ and $F_x$ denotes the fractional field of $\calO_x$ (in the note, $\calG(\calO_x)$ will be a Moy-Prasad subgroup so $P_x$ will be a parahoric subgroup of $G(F_x)$). Then $Q_x:=P_x/K_x$ acts on $\Bun_\calG$ by translations and therefore there is the moment map $\mu_x: T^*\Bun_\calG\to \frakq_x^*$ where $\frakq_x$ is the Lie algebra of $Q_x:=P_x/K_x$. Then one has the enhanced Hitchin map
\[h^{cl}\times\prod_x \mu_x: T^*\Bun_\calG\to \on{Hitch}(X)_\calG\times \prod_x \frakq_x^*.\]
Our observation is that in some nice cases (e.g. $K_x$ is the pro-unipotent radical  of a pararhoric  or $K_x$ is a subgroup of $G(F_x)$ considered in \cite{RY}, see \S~\ref{loc Hitchin}), one can obtain a maximal Possion commuting subalgebra of the ring of regular functions on $T^*\Bun_\calG$ from this enhanced system. Then we can quantize this system  by (a generalization of) Beilinson-Drinfeld's machinery to obtain an affine scheme $\Spec A$ and Hecke eigensheaves parameterized by it. A new feature is that points on $\Spec A$ do not parameterize opers in general, but by construction there is a map from $\Spec A$ to the space of opers. E.g., in the case considered in \S~\ref{Proof}, $\Spec A$ paramterizes linear forms on $\Lie(I(1)/I(2))$.

Let us summarize the contents of the note. 
We first review of the main results of
\cite{BD} in \S~\ref{KMCA}-\S~\ref{recon}. We take the opportunity to describe
a slightly generalized version of \cite{BD}
in order to deal with the level structures (in particular see Corollary \ref{variant}). These results are probably known to experts but seem not to exist in literature yet.
In \S~\ref{loc Hitchin}, we study some geometry of the local Hitchin map by relating the invariant theory of simple Lie algebras over local fields with the invariant theory of Vinberg's theta groups. In particular, we determine the images of the dual of some subalgebras of a loop algebra under this map. Then we quantize these results in \S~\ref{endo alg} to get some information of the endomorphism algebras of some modules of an affine Kac-Moody algebra at the critical level. These two sections contain results more than needed in the proof of the main theorem. It may be useful to relate the subsequent  works \cite{Y} and \cite{C}. 
The main result is proved in \S~\ref{Proof}.

\medskip\noindent\bf Notations. \rm
In the note, $X$ will be a smooth projective curve over $\bbC$. For
every closed point $x\in X$, let $\calO_x$ denote the complete local
ring of $X$ at $x$ and let $F_x$ denote its fractional field. Let
$D_x=\Spec \calO_x$ and $D_x^\times=\Spec F_x$. For a scheme $T$ of finite type over $\calO_x$, we denote by $T(\calO_x)$ the jet space $T$. For an affine scheme $T$ of finite type over $F_x$, we denote by $T(F_x)$ the ind-scheme of the loop space of $T$. 

In the note, unless otherwise specified, $\calG$ is a
(fiberwise) connected smooth affine group scheme over $X$ such that
$G=\calG\otimes \bbC(X)$ is reductive. We denote by $\Bun_\calG$ the stack of $\calG$-bundles on $X$. This is a smooth algebraic stack over $\bbC$.

In the note, a presheaf means a covariant functor from the category of commutative $\bbC$-algebras to the category of sets. 

Let $\calF\to\calF'$ be a morphism of presheaves and let $u:\Spec R\to \calF'$ be morphism. We use $\calF\otimes R$, or $\calF_R$ to denote $\calF\times_{\calF'}\Spec R$ if no confusion is likely to arise.

For
an affine (ind-)scheme $T$, we denote by $\on{Fun} T$ or $\bbC[T]$ the
(pro-)algebra of regular functions on $T$. 

\medskip\noindent\bf Acknowledgement. \rm The author thanks the referee for careful reading and critical questioning of the early version of the note. The author also thanks T.-H. Chen and M. Kamgarpour for very useful comments. The work is partially supported by NSF grant DMS-1001280/1313894 and DMS-1303296/1535464 and the AMS Centennial Fellowship.

\section{Kac-Moody chiral algebras}\label{KMCA}
In the following two sections, we review the main results of the work \cite{BD}, with some generalizations in order to deal with level structures.
\subsection{Beilinson-Drinfeld Grassmannians}
We refer to \cite[\S~3]{Z2} for a general introduction.
For $x\in X$, let $\Gr_{\calG,x}$ denote the
affine Grassmannian of $\calG$ at $x$. This is a presheaf that assigns to every $\bbC$-algebra $R$ the pairs
$(\calF,\beta_x)$, where $\calF$ is a $\calG$-torsor on $X\otimes R$ and
$\beta_x$ is a trivialization of $\calF$ away from the graph $\Gamma_x$ of $x$. Then $\Gr_{\calG,x}$ is represented by an ind-scheme over $\bbC$, ind of finite type, and we can write
$$\Gr_{\calG,x}\simeq G(F_x)/K_x,$$ where $K_x=\calG(\calO_x)$, regarded as a pro-algebraic group over $\bbC$. 
By forgetting $\beta_x$, we get a map
$$u_x: \Gr_{\calG,x}\to \Bun_\calG.$$

By allowing multiple points and allowing points to move, one obtain the Ran version of the Beilinson-Drinfeld Grassmannian $\Gr_{\calG, \Ran(X)}$, which is a presheaf\footnote{Note that it is not a sheaf.} that assigns to every $\bbC$-algebra $R$, the triples $(\{x_i\}, \calF,\beta)$, where $\{x_i\}\subset X(R)$ is a finite non-empty subset, $\calF$ is a $\calG$-torsor on $X\otimes R$, and $\beta$ is a trivialization of $\calF$ away from the union of the graphs $\cup_i \Gamma_{x_i}$ of $\{x_i\}$.  There is a natural projection
\[q: \Gr_{\calG,\Ran(X)}\to \Ran(X),\]
to the Ran space $\Ran(X)$ of $X$, which assigns to every $\bbC$-algebra $R$, the finite non-empty subsets of $X(R)$. If $U\subset X$ is an open subset, we have the Ran space $\Ran(U)$ of $U$, which is an open sub-presheaf of $\Ran(X)$. Let $\Gr_{\calG,\Ran(U)}:=\Gr_{\calG,\Ran(X)}\times_{\Ran(X)}\Ran(U)$ denote the base change. There is the (ad\`elic version) uniformization map
$$u_{\Ran}: \Gr_{\calG,\Ran(U)}\to \Bun_\calG,$$
by forgetting $(\{x_i\},\beta)$, and a unit section 
\[e_\Ran:\Ran(U)\to \Gr_{\calG,\Ran(U)}\]
given by the canonical trivialization of the trivial $\calG$-torsor. In the sequel, if $\calG$ and $U$ are clear from the context, we write $\Ran(U)$ by $\Ran$ and $\Gr_{\calG,\Ran(U)}$ by $\Gr_\Ran$ for brevity.

A salient feature of the Beilinson-Drinfeld Grassmannian is its factorizable property (cf. \cite[\S~3.10.16]{BD1}). First, the Ran space $\Ran=\Ran(U)$ has a semi-group structure given by union of points
\[\on{union}: \Ran\times \Ran\to \Ran,\quad (\{x_i\},\{x_j\})\mapsto \{x_i,x_j\}.\]
Let $(\Ran\times \Ran)_{disj}$ denote the open sub presheaf consisting of those $\{x_i\}$ and $\{x_j\}$ with $\{x_i\}\cap \{x_j\}=\emptyset$. Then the factorization property amounts to a canonical isomorphism
\[ \Gr_{\Ran}\times_{\Ran}(\Ran\times\Ran)_{disj}\simeq \Gr_{\Ran}\times\Gr_{\Ran}|_{(\Ran\times\Ran)_{disj}},\]
compatible with the unit section $e:\Ran\to\Gr_\Ran$ and satisfying a natural cocycle condition over $(\Ran\times\Ran\times \Ran)_{disj}$. 

Next, we recall a factorizable line bundle on $\Gr_\Ran$.
Let $\omega_{\Bun_\calG}$ be the canonical sheaf of
$\Bun_\calG$. This is in fact a line bundle that can be described as follows. Let $\calE$ denote the universal $\calG$-torsor on $X\times\Bun_\calG$ and $\ad\calE=\calE\times^\calG \Lie\calG$ be the adjoint bundle.
Then $\omega_{\Bun_\calG}$ is isomorphic to the determinant line bundle $\det(\ad\calE)$ for $\ad\calE$. Its fiber over the trivial bundle is $\det(\Lie \calG)$. Therefore, we can normalize $\omega_{\Bun_\calG}$ to a rigidified line bundle as
\[\omega_{\Bun_\calG}^{\sharp}=\omega_{\Bun_{\calG}} \otimes \det(\Lie \calG)^{-1}.\]
Let $\calL_{2c,x}=u_x^*\omega_{\Bun_\calG}^\sharp\in \Pic(\Gr_{\calG,x})$, and let $\calL_{2c,\Ran}=u_{\Ran}^*\omega_{\Bun_\calG}^\sharp$.

\begin{lem}\label{fl}
The line bundle $\calL_{2c,\Ran}$ admits a natural factorizable structure, compatible with the factorizable structure of $\Gr_\Ran$. 
\end{lem}
\begin{proof}
Let $\calG'=\GL(\Lie \calG)$. This is an inner form of $\GL_n$ on $X$. Let $\calL_{\det}$ denote the determinant line bundle on $\Bun_{\calG'}$ as well as its pullback to $\Gr_{\calG',\Ran(X)}$. Then $\calL_{2c}$ is just the pullback of $\calL_{\det}$ along the natural morphism $\Ad:\Gr_{\calG,\Ran(X)}\to\Gr_{\calG',\Ran(X)}$, sending a $\calG$-torsor to its adjoint bundle. Note that $\ad$ is compatible with the natural factorizable structures on $\Gr_{\calG,\Ran(X)}$ and on $\Gr_{\calG',\Ran(X)}$. Therefore, it is enough to show that $\calL_{\det}$ admits a natural (graded) factorizable structure, which is well-known (e.g. see \cite[Remark 3.2.8]{Z2}).
\end{proof}

\subsection{Kac-Moody chiral algebras}\label{KM chiral}
Let $F=\bbC(X)$. We assume that $\calG$ is a smooth, affine fiberwise connected group scheme over $X$ with $G:=\calG\otimes F$ simple and simply-connected. 
Let us briefly discuss the Kac-Moody chiral algebra attached to $\calG$, generalizing the usual Kac-Moody chiral algebra for the constant group scheme over $X$.
  
Let $q: \Gr_{\calG}:=X\times_{\Ran(X)}\Gr_{\calG,\Ran(X)}\to X$ be the Beilinson-Drinfeld Grassmannian over $X$, and $e:X\to\Gr_{\calG}$ the unit section given by the trivial
$\calG$-torsor. Let $\calL_{2c}$ be the pullback of the line bundle $\calL_{2c,\Ran}$ to
$\Gr_{\calG}$. Let $\delta_e$ be the $\calL_{2c}^{1/2}$-twisted delta right $D$-module along the unit section\footnote{Recall that given a line bundle $\calL$ on an algebraic variety $X$, it makes sense to talk about $\calL^\la$-twisted $D$-modules on $X$ for any $\la\in\bbC$.}. Let us define
\[\calV ac_X:=q_*(\delta_e),\]
the push-forward of $\delta_e$ as a quasi-coherent sheaf.
\begin{lem}\label{flat}
The sheaf
$\calV ac_X$ is flat as an $\calO_X$-module.
\end{lem}
\begin{proof}First note that $\pi:\Gr_\calG\to X$ is formally smooth. Indeed, let $\calL^+\calG$ (resp. $\calL\calG$) be the global jet (resp. loop) group of $\calG$, which classifies $(x,\gamma)$, where $x\in X$ and $\gamma: D_x\to \calG$ (resp. $\gamma: D_x^\times\to \calG$) is a section. Then $\calL^+\calG$ and $\calL\calG$ are formally smooth. As $\Gr_\calG=\calL\calG/\calL^+\calG$, it is formally smooth.

Then one can argue as \cite[\S~7.11.8]{BD} to show that the tangent space $T_e\Gr_\calG$ of $\Gr_\calG$ along $e:X\to \Gr_\calG$ is a flat $\calO_X$-module, and $\delta_e$ admits a filtration whose associated graded is $\on{Sym}T_e\Gr_\calG$. Therefore, $\calV ac_X$ is flat.
\end{proof}
We refer to \cite{BD1} for the generalities of chiral algebras.
\begin{lem}\label{chiral}
$\calV ac_X$ has a natural chiral algebra structure over $X$.
\end{lem}
\begin{proof}For each finite non-empty set $I$, we have $\Gr_{\calG,X^I}\to X^I$ and therefore the similarly defined $\calV ac_{X^I}$.
Using the dictionary between chiral algebras and factorization algebras \cite[Chap. 2, Sect. 4]{BD1}, the lemma is equivalent to saying that there is a natural factorization structure on the collection $\{\calV ac_{X^I}\otimes\omega_{X^I}^{-1}\}$. This follows from Lemma \ref{fl}.
\end{proof} 

We call $\calV ac_X$ the Kac-Moody chiral algebra associated to $\calG$. 

\subsection{Vacuum modules}\label{vac mod}
 By flatness, the fiber of $\calV ac_X$ over a point $x\in X$ is
\[\Vac_x:=\Gamma(\Gr_{\calG,x},\delta_e|_{\Gr_{\calG,x}}).\]
We need a more representation theoretical description of $\Vac_x$ in some cases. 
For the purpose, we need to fix a few notations.

We write $\frakg_x=\Lie \calG\otimes F_x$, regarded as an infinite dimensional Lie algebra over $\bbC$. Let $\hat{\frakg}_x$ be the ``level one" Kac-Moody central extension of $\frakg_x$ (i.e., the completion of the derived subalgebra of one of those in \cite[Theorem 7.4, 8.3]{Ka}). We fix an Iwahori subgroup $I\subset G(F_x)$, and let $\Lambda_i, i=0,\ldots,\ell$ be the fundamental weights of $\hat{\frakg}_x$ (with respect to $I$). 
Let $\rho_{\on{aff}}=\sum \Lambda_i$. For a standard parahoric subgroup $P\supset I$ of $G(F_x)$, let  $\rho_P$ denote the half sum of affine roots in $P/I$. As before, we write $K_x=\calG(\calO_x)$, and let $\frakk_x$ denote its Lie algebra. Its pre-image in $\hat{\frakg}_{x}$ is denoted by $\hat{\frakk}_{x}$.

\begin{lem}\label{vacuum I}
Assume that $K_x=P$ is a standard parahoric subgroup. Then as a representation of $\hat{\frakg}_x$,
\[\Vac_x=\on{Ind}_{\hat{\frakk}_x}^{\hat{\frakg}_x} (\bbC_{-\rho_{\on{aff}}+\rho_P}),\]
where $\bbC_{-\rho_{\on{aff}}+\rho_P}$ is the $1$-dimensional representation of $\hat{\frakk}_x$ given by the character $-\rho_{\on{aff}}+\rho_P$.
\end{lem}
\begin{proof}The the lemma follows from the calculation of \cite[\S~4.1]{Z}. Note that the line bundle $\calL_{2c,x}$ here corresponds to $\calL_{2c}^{-1}$ in \emph{loc. cit.}.
\end{proof}

As explained in \cite[\S~4.1]{Z}, the restriction of $-\rho_{\on{aff}}+\rho_P$ to the center of $\hat{\frakg}_x$ is always given by multiplication by $-h^\vee$, the dual Coxeter number of (the split form of) $G$. Therefore, it is customary to consider $\Vac_x$ as a module of the critical extension $\hat{\frakg}_{c,x}$ of $\frakg_x$, i.e. $(-h^\vee)$-multiple of the ``level one" central extension $\hat{\frakg}_x$, so that the center of $\hat{\frakg}_{c,x}$ acts on $\Vac_x$ by identity. In the sequel, we adapt this latter point of view. Let $\mathbf{1}_x$ (or $\mathbf{1}$ for simplicity) denote the central element in $\hat{\frakg}_{c,x}$ that acts on $\Vac_x$ by identity.

When $\Hom(K_x,\bbG_m)=1$, e.g. $K_x$ is a maximal parahoric or a pro-unipotent group, there is another expression of $\Vac_x$, which might justify the notation. Namely, in this case, there is a unique splitting of  the Lie algebra $\hat{\frakk}_x\simeq \bbC\mathbf{1}\oplus \frakk_x$ coming from the splitting at the group level. Then
\begin{equation}\label{vacuum II}
\Vac_x=\on{Ind}_{\bbC\mathbf{1}\oplus \frakk_x}^{\hat{\frakg}_{c,x}} (\bbC),
\end{equation}
where $\mathbf{1}$ acts on $\bbC$ by identity and $\frakk_x$ acts trivially. That is, $\Vac_x$ is a ``vacuum" module.

\begin{rmk}\label{diff splitting}
However, if $\Hom(K_x,\bbG_m)$ is non-trivial, there is no ``natural" splitting of $\hat{\frakk}_x$. For example, if $K_x=I$ is the Iwahori subgroup, by embedding $I$ into different maximal parahorics one obtains \emph{different} splittings of $\hat{\frakk}_x$, and if one writes $\Vac_x$ as the form $\on{Ind}_{\bbC\mathbf{1}\oplus \frakk_x}^{\hat{\frakg}_{c,x}} (\bbC_\chi)$ using these splittings, $\frakk_x$ will act on $\bbC_\chi$ via different characters $\chi$.
\end{rmk}

\subsection{The center}For any chiral algebra $\calA$  over a curve, one can associate the
algebra of its endomorphisms, denoted by $\calE nd(\calA)$. As
a sheaf of $\calO_X$-modules on $X$, it is defined as
\[\calE nd(\calA)=\calH om_\calA(\calA,\calA),\]
where $\calH om_{\calA}$ is taken in the category of chiral $\calA$-modules. If $\calA$ is  $\calO_X$-flat, which is the case we will be considering, it is quasi-coherent. 
Obviously, $\calE nd(\calA)$ is an algebra by composition. Less
obviously, there is a natural chiral algebra structure on $\calE
nd(\calA)\otimes\omega_X$ which is compatible with the algebra
structure. In fact, the natural map $\calE nd(\calA)\otimes\omega_X\to \calA$ identifies $\calE nd(\calA)\otimes\omega_X$ with the center $\frakz(\calA)$ of $\calA$.

Therefore, $\calE nd(\calA)$ is a commutative
$\calD_X$-algebra. There is a natural
injective mapping $\calE nd(\calA)_x\to \End(\calA_x)$ which is not
necessarily an isomorphism in general, where $\End(\calA_x)$ is the
endomorphism algebra of $\calA_x$ as a chiral $\calA$-module. However,
this is an isomorphism if there is some open neighborhood $U$
containing $x$ such that $\calA|_U$ is constructed from a vertex
algebra (in the sense of \cite[Chap. 19]{FB}). We refer to \cite{N} for details of the above discussion. We apply the above discussion to $\calE nd(\calV ac_X)$. In particular, if $\calG$ is unramified at $x$, $\calE nd(\calV ac_X)_x\simeq \End(\Vac_x)$.

To continue, we assume that $G=\calG\otimes F$ is split, i.e. it is the base change of simple complex Lie group $G_0$ from $\bbC$ to $F$.  The reason we make this assumption is that the relevant theory for twisted affine algebras has not been fully developed.
We denote by $\frakg$ the Lie
algebra of $G$ and ${^L}\frakg$ the Langlands dual Lie algebra, equipped with a Borel subalgebra ${^L}\frakb$. Following the usual convention, we consider ${^L}\frakg$ and ${^L}\frakb$ as finite dimensional Lie algebras over $\bbC$ (while $\frakg$ is a Lie algebra over $F$).

Recall the definition of opers (cf. \cite[\S~3]{BD}). Let $\on{Op}_{{^L}\frakg}(D_x)$ (resp.
$\on{Op}_{{^L}\frakg}(D^\times_x)$) be the scheme (resp. ind-scheme) of
${^L}\frakg$-opers on the disc $D_x$ (resp. punctured disc $D^\times_x$).  Let $A_{{^L}\frakg}(D_x)$ (resp. $A_{{^L}\frakg}(D_x^\times)$) denote the ring of its regular functions.
Let $\hat{U}_{c}(\frakg_x)$ denote the completed universal enveloping algebra of $\hat{\frakg}_{c,x}$, and let $\frakZ_x$ denote the center of $\hat{U}_{c}(\frakg_x)$. Then  the Feigin-Frenkel
isomorphism (\cite[\S~3.7.13]{BD}, \cite{F}) asserts that there is a canonical isomorphism
\begin{equation}\label{FeFr}
\varphi_x:A_{{^L}\frakg}(D^\times_x)\simeq \frakZ_x, 
\end{equation}
satisfying certain compatibility conditions (one of which will be reviewed in \S~\ref{endo alg}).
Then we obtain
 \[A_{{^L}\frakg}(D^\times_x)\simeq\frakZ_x\to\End(\on{Vac}_x).\]
If $\calG$ is unramified at $x$, this map factors as
\[A_{^{L}\frakg}(D_x^\times)\twoheadrightarrow A_{^{L}\frakg}(D_x)\simeq\frakz_x:=\End(\Vac_x).\]

Let $U\subset X$ be an open subscheme such that $\calG|_U\simeq G_0\times U$. Let $\on{Op}_{{^L}\frakg}|_U$ be the $\calD_U$-scheme over $U$, whose
fiber over $x\in U$ is the scheme of ${^L}\frakg$-opers on $D_x$.
Then by the above generality, the Feigin-Frenkel
isomorphism gives rise to
\[\Spec \calE nd(\calV ac_U)\simeq \on{Op}_{{^L}\frakg}|_U.\]

Recall that for a commutative $\calD_U$-algebra $\calB$, one can take
the algebra of its horizontal sections $H_\nabla(U,\calB)$ (or
so-called conformal blocks) \cite[\S~2.6]{BD}, which is usually a
topological commutative algebra. For example,
\[\Spec
H_{\nabla}(U,\on{Op}_{{^L}\frakg})=\on{Op}_{{^L}\frakg}(U)\] is the
ind-scheme of ${^L}\frakg$-opers on $U$ (\cite[\S~3.3]{BD}). 
As the map
$$H_{\nabla}(U,\calE nd(\calV ac_U))\to H_{\nabla}(X,\calE nd(\calV
ac_X))$$ is surjective, we have a closed embedding
\[\Spec H_{\nabla}(X,\calE nd(\calV ac_X))\to\on{Op}_{{^L}\frakg}(U).\]
Let $\Op_{{^L}\frakg}(X)_\calG$ denote the image of this closed
embedding. This is a subscheme (rather than an ind-scheme) of
$\on{Op}_{{^L}\frakg}(U)$.

We recall the characterization $\on{Op}_{{^L}\frakg}(X)_{\calG}$.
\begin{lem}\label{support}
Let $X\setminus U=\{x_1,\ldots,x_n\}$.
Assume that the support of $\Vac_{x_i}$ (as an
$\frakZ_{x_i}$-module) is
$Z_{x_i}\subset\on{Op}_{{^L}\frakg}(D^\times_{x_i})$, i.e.
$\on{Fun}(Z_{x_i})=\on{Im} (A_{{^L}\frakg}(D_{x_i}^\times)\to
\End(\Vac_{x_i}))$. Then $$\on{Op}_{{^L}\frakg}(X)_{\calG}\simeq
\on{Op}_{{^L}\frakg}(U)\times_{\prod_i\on{Op}_{{^L}\frakg}(D^\times_{x_i})}\prod
Z_{x_i}.$$
\end{lem}
\begin{proof}Note that $\calV ac_X$ is flat, and therefore by Lemma \ref{flat}, $\calE nd(\calV ac_X)_{x_i}\to \calE nd(\Vac_{x_i})$ is injective. Recall that $\Spec \calE nd(\calV ac_X)_{x_i}$ is identified with the scheme of horizontal sections of $\calE nd(\calV ac_X)$ on the disc $D_{x_i}$. On the other hand, the ind-scheme of horizontal sections of $\calE nd(\calV ac_X)$ on $D_{x_i}^\times$ is just $\on{Op}_{{^L}\frakg}(D_{x_i}^\times)$. Therefore $A_{{^L}\frakg}(D_{x_i}^\times)\to
\calE nd(\calV ac_X)_{x_i}$ is surjective.
Therefore,
 $$\calE nd(\calV ac_X)_{x_i}=\on{Im} (A_{{^L}\frakg}(D_{x_i}^\times)\to
\End(\Vac_{x_i})).$$ Then
the lemma follows from the last paragraph of \cite[\S~2.4.12]{BD1} applied to the commutative chiral algebra $\calE nd(\calV ac_X)$.
\end{proof}

In \S~\ref{endo alg}, we will give an explicit description of $Z_{x_i}$ in some cases.

\section{Hecke eigensheaves}\label{recon} 
\subsection{Square root of $\omega_{\Bun_\calG}^\sharp$}
In order to construct the Hecke eigensheaves, it is important to have a square root $\omega_{\Bun_\calG}^{1/2}$ of $\omega_{\Bun_\calG}^\sharp$ as a line bundle. In the case considered in \S~\ref{Proof}, the existence of such a square root is very easy (see Remark \ref{sqr5}). But in general, this is subtle. The case when $\calG$ is a constant semisimple group scheme is discussed in \cite[\S~4]{BD}. Here, we describe a result for the case when $\calG$ is non-constant, but $G=\calG\otimes\bbC(X)$ is simple and simply-connected. The readers can skip this subsection.

We write $F=\bbC(X)$. We make the following assumptions on $\calG$.
\begin{enumerate}
\item[(i)] $G=\calG\otimes F$ is absolutely simple and simply-connected; 
\item[(ii)] for every $x\in X$, either $\calG|_{\calO_x}$ is a parahoric group scheme or the fiber of $\calG$ at $x$ is a unipotent group. 
\end{enumerate}

Recall the line bundle $\calL_{2c,x}=u_x^*\omega_{\Bun_\calG}^\sharp$.

\begin{prop}\label{square root}
Assumptions are as above. Then a square root of $\omega_{\Bun_{\calG}}^\sharp$ exists as a line bundle if and only if for every $x$ such that $\calG|_{\calO_x}$ is a parahoric group, the line bundle $\calL_{2c,x}$ admits a square root as a line bundle. In addition, such a square root is unique.
\end{prop}
We denote the unique square root, rigidified at the trivial $\calG$-torsor, by $\omega_{\Bun_\calG}^{1/2}$.

The ``only if" part is trivial. Note that there are examples of parahorics of $G(F_x)$ such that $\calL_{2c,x}$ does not admit a square root (cf. \cite[Remark 6.1]{Z}). The ``if" part and the uniqueness is based on the following proposition.

\begin{prop}Let $U\subset X$ be an open subset such that for every $x\in X\setminus U$, the fiber of $\calG$ at $x$ is unipotent. Then 
the pullback functor $u_{\Ran}^*$ induces an equivalence of categories between rigidified line bundles on $\Bun_\calG$ and on $\Gr_{\calG,\Ran(U)}$.
\end{prop}
This statement is a slight variant/generalization of \cite[4.3.14]{BD} or \cite[4.9.1]{BD1}. 
It allows us to reduce the problem to construct a square root of $\calL_{2c,\Ran}$. Then using the factorization property, we can work \'{e}tale locally around a point. We omit the details (but see \cite[\S~4.2]{Z2}) since we do not really use Proposition \ref{square root} in \S~\ref{Proof}.

\subsection{Construction of Hecke eigensheaves}
We further assume that $G:=\calG\otimes F=G_0\otimes_{\bbC} F$, where $G_0$ is a simple, simply-connected complex Lie group. 

Now we assume that $\Bun_\calG$ is ``good" in the sense of
Beilinson-Drinfeld, i.e.
$$\dim T^*\Bun_\calG=2\dim\Bun_\calG.$$
In this case one can construct the D-module of the sheaf of
critically twisted (a.k.a. $\omega_{\Bun_\calG}^{1/2}$-twisted)
differential operators on the smooth site $(\Bun_\calG)_{sm}$ of
$\Bun_\calG$, denoted by $\calD'$. Let $D'=(\underline{\End}
\calD')^{op}$ be the sheaf of endomorphisms of $\calD'$ as a twisted
D-module. Then $D'$ is a sheaf of associative algebras on
$(\Bun_\calG)_{sm}$ and $D'\simeq (D')^{op}$. For more details, we
refer to \cite[\S~1]{BD}. \quash{We just mention that if $U\subset \Bun_\calG$ is a (smooth) open subscheme, then $\calD'|_U\simeq \calD_U\otimes \omega_{\Bun_\calG}^{1/2}|_U$, where $\calD_U$ is the usual sheaf of differential operators on $U$.}

Let $\Bun_{\calG,x}$ be the scheme classifying pairs
$(\calF,\beta)$, where $\calF$ is a $\calG$-torsor on $X$ and
$\beta$ is a trivialization of $\calF$ on $D_x=\Spec \calO_x$. It
admits a $(\hat{\frakg}_{c,x},K_x)$-action, and
$\Bun_{\calG,x}/K_x\simeq \Bun_\calG$. Now applying the standard
localization construction to the Harish-Chandra module $\on{Vac}_x$
(cf. \cite[\S~1]{BD}) gives rise to
\[\on{Loc}(\on{Vac}_x)\simeq\calD'\]
as critically twisted $D$-modules on $\Bun_\calG$. It induces a natural ring homomorphism
\[h_x: \frakZ_x\to \End(\Vac_x)\to \Gamma(\Bun_\calG, D').\]
In fact, the mappings $h_x$ can be organized into a horizontal morphism $h$
of $\calD_X$-algebras over $X$ (we refer to \cite[\S~2.6]{BD} for
the generalities of $\calD_X$-algebras and \cite[\S~2.8]{BD} for the horizontality of $h$),
\begin{equation}\label{hglob}
h: \calE nd(\calV ac_X)\to\Gamma(\Bun_\calG,D')\otimes\calO_X.
\end{equation}
\quash{Namely, by varying $x$, we get a scheme $\Bun_{\calG, X}$ over $X$ classifying a $\calG$-torsor $\calF$, a point $x\in X$ and a trivialization of $\calF$ along the graph $\Gamma_x$ of $x$. According to \cite[\S~2.8]{BD}, this is a crystal of schemes over $X$, and the natural map $\Bun_{\calG,X}\to \Bun_\calG\times X$ by forgetting the trivialization is 

 Indeed, the construction of $h$ is as in \cite[\S~2.8]{BD}, where the horizontality of $h|_U$ is also proved. But if there is an algebra homomorphism of $\calD_X$-algebras that is horizontal over some open subscheme $U\subset X$, then this homomorphism is horizontal over the whole curve.}

By  taking the horizontal sections, one obtains a mapping
\begin{equation}\label{hhorizontal}
h_{\nabla}:H_{\nabla}(X,\calE nd(\calV
ac_X))\to\Gamma(\Bun_{\calG},D').
\end{equation}
Therefore, \eqref{hhorizontal} can be rewrite as a mapping
\begin{equation}\label{hhorizontal1}h_{\nabla}:\on{Fun}\on{Op}_{{^L}\frakg}(U)\twoheadrightarrow\on{Fun}\on{Op}_{{^L}\frakg}(X)_{\calG}\to
\Gamma(\Bun_{\calG},D').\end{equation}

The mapping \eqref{hhorizontal1} is a quantization of a classical
Hitchin system. Namely, as explained in \cite[\S
3.1.13]{BD} (or see \S~\ref{endo alg} for a review), there is a natural filtration on the algebra $\on{Fun}\on{Op}_{{^L}\frakg}(U)$ whose
associated graded is the algebra of functions on the classical
Hitchin space
\begin{equation}\label{Hitchin base}
\on{Hitch}(U)=\Gamma(U,\frakc^*\times^{\bbG_m}\omega_U^\times).
\end{equation}
Here, using the fact that $\calG$ is unramified over $U$, we regard 
$$\frakc^*=\frakg^*/\!\!/G=\Spec F[\frakg^*]^G$$ as a scheme over $U$. 
On the other hand, there is a natural
filtration on $\Gamma(\Bun_\calG,D')$ coming from the order of the
differential operators. Then \eqref{hhorizontal1} is strictly
compatible with the filtration and the associated graded map gives
rise to the classical Hitchin map
\[h^{cl}: T^*\Bun_\calG\to \on{Hitch}(U).\]
Its image is the closed
subscheme $\on{Hitch}(X)_\calG\subset\on{Hitch}(U)$ whose algebra of
functions is the associated graded of $\on{Fun}\on{Op}_{{^L}\frakg}(X)_\calG$.

Let ${^L}G$ be the Langlands dual group of $G$, which is of adjoint type. The following theorem summarizes the main result of \cite{BD}. 
\begin{thm}\label{BDmain}
Assume that the line bundle $\omega_{\Bun_\calG}^{1/2}$ exists.
Let $\chi\in
\on{Op}_{{^L}\frakg}(X)_{\calG}\subset\on{Op}_{{^L}\frakg}(U)$ be a
closed point, which gives rise to a ${^L}\frakg$-oper $\calE$ on
$U$. Let
$$\varphi_\chi:\on{Fun}\on{Op}_{{^L}\frakg}(X)_{\calG}\to\bbC$$ be the
corresponding homomorphism of $\bbC$-algebras. Then
\[\Aut_{\calE}:=(\calD'\otimes^L_{\on{Fun}\on{Op}_{{^L}\frakg}(X)_{\calG},\varphi_\chi}\bbC)\otimes\omega_{\Bun_\calG}^{-1/2}\]
is a Hecke-eigensheaf on $\Bun_\calG$ with respect to $\calE$
(regarded as a ${^L}G$-local system).
\end{thm}

\begin{rmk}The statement of the about theorem is weaker than the
main theorem in \cite{BD}, where $\calG$ is the
constant group scheme (the unramified case). In this case Beilinson-Drinfeld proved that
$\on{Op}_{{^L}\frakg}(X)_{\calG}=\on{Op}_{{^L}\frakg}(X)$ is the
space of ${^L}\frakg$-opers on $X$,  that
$\on{Fun}\on{Op}_{{^L}\frakg}(X)\simeq\Gamma(\Bun_G,D')$, and that $\calD'$ is flat over $\on{Fun}\on{Op}_{{^L}\frakg}(X)$. It then follows that in the definition of $\Aut_\calE$, the derived tensor product can be replaced by the underived tensor product, and that $\Aut_\calE$ is a non-zero
holonomic D-module.

The proof of these assertions relies on the fact that the
classical Hitchin map is a completely integrable system. However, for general level structures, we do not know whether $\calD'$ is always flat over $\on{Fun}\on{Op}_{{^L}\frakg}(X)_{\calG}$. In addition,
the automorphic D-modules constructed as above will not be holonomic, as $\on{Fun}\on{Op}_{{^L}\frakg}(X)_{\calG}\subset\Gamma(\Bun_\calG,D')$ might be too small. For example, this will be the case for the group scheme $\calG=\calG(0,2)$ introduced in \S~\ref{Proof}. 
\end{rmk}

Sometimes, we can remedy the situation by replacing $\on{Fun}\on{Op}_{{^L}\frakg}(X)_\calG$ by a larger commutative subalgebra in $\Gamma(\Bun_\calG,D')$. To state our result, we first give a brief review of the local ingredient needed in the proof of Theorem \ref{BDmain}. Let $x\in U\subset X$, so $\calG$ is unramified at $x$. Then $\Gr_{\calG,x}$ is the usual affine Grassmannian $\Gr_{G_0,x}$ of $G_0$ at $x$. Let $\on{Sph}_x=\on{D}\mbox{-mod}(\Gr_{G_0,x})^{G_0(\calO_x)}$ be the Satake category of $G_0(\calO_x)$-equivariant D-modules on $\Gr_{G_0,x}$, and let $\calS_x: \on{Rep}({^L}G)\to \on{Sph}_x$ denote the geometric Satake equivalence. As explained in \cite[\S~7.8]{BD}, there is a convolution action $\star$ of $\on{Sph}_x$ on the derived category of Harish-Chandra $(\hat{\frakg}_{c,x},G_0(\calO_x))$-modules.
The local Hecke eigen property of $\Vac_x$ (see \cite[\S~5.5]{BD}) asserts that for every $V\in \on{Rep}({^L}G)$ there is
a canonical isomorphism of Harish-Chandra $(\hat{\frakg}_{c,x},K_x)$-modules,
\begin{equation}\label{local Hecke}
 \calS_x(V)\star \Vac_x\simeq (\calE_x)_V \otimes_{\frakz_x} \Vac_x,
\end{equation}
where $\calE_x$ is the fiber over $x\in D_x$ of the universal ${^L}G$-torsor on $\on{Op}_{{^L}\frakg}(D_x)\hat{\times}D_x$, and $(\calE_x)_V$ is its twist by $V$, regarded as a finite projective $\frakz_x\simeq A_{{^L}\frakg}(D_x)$ module. In particular, the isomorphism commutes with the actions of $\frakz_x$ on both sides.
The localization functor intertwines the action of $\on{Sph}_x$ on the derived category of $(\hat{\frakg}_{c,x},G_0(\calO_x))$-modules and on the derived category of D-modules on $\Bun_\calG$. Then by letting $x$ move in $U$, one obtains a canonical isomorphism
\begin{equation}\label{global Hecke}
\calS_U(V)\star \calD'\simeq  \calE_V\otimes_{A_{{^L}\frakg}(U)} \calD',
\end{equation}
of D-modules on $\Bun_{\calG}\times X$, compatible with the actions of $A_{{^L}\frakg}(U)$ on both sides. Here $\calS_U$ is the family version of $\calS_x$ over $U$.

Now let $y\in X\setminus U$, and let $A_y\subset \End(\Vac_y)$ be a commutative subalgebra.
We regard $\Vac_x\otimes \Vac_y$ as the representation of $\hat{\frakg}_{c,x,y}:=\hat{\frakg}_{c,x}\oplus\hat{\frakg}_{c,y}/(\mathbf{1}_x-\mathbf{1}_y)$. Then by tensoring with $\Vac_y$, \eqref{local Hecke} induces a canonical isomorphism of Harish-Chandra $(\hat{\frakg}_{c,x,y},K_x\times K_y)$-modules,
\[
\calS_x(V)\star \Vac_x\otimes \Vac_y\simeq (\calE_x)_V \otimes_{\frakz_x} \Vac_x\otimes \Vac_y,
\]
compatible with the $A_y\otimes \frakz_x$-module structures. By the (two-points version) localization functor and allowing $x$ to move in $U$ (but fixing $y$), we see that the isomorphism \eqref{global Hecke} is compatible with the $A_y\otimes A_{{^L}\frakg}(U)$-module structures on both side. 

Now for every $y\in X\setminus U$, let $A_y$ be a commutative subalgebra $\End(\Vac_y)$ containing the image of $\frakZ_y\to \End(\Vac_y)$. Then the map in \eqref{hhorizontal1} can be upgraded to a map
$$A_{{^L}\frakg}(U)\otimes \prod_y A_y\to  \Gamma(\Bun_\calG, D').$$ 
Let $A$ denote its image. Then $A$ is a commutative subalgebra of $\Gamma(\Bun_\calG, D')$ containing $\on{Fun}\on{Op}_{{^L}\frakg}(X)_\calG$. Let $\pi: \Spec A\to \Op_{{^L}\frakg}(X)_\calG$ denote the corresponding map of schemes induced by this inclusion.
\begin{cor}\label{variant}
Assume that $\calD'$ is $A$-flat.
Then for every $\chi\in \Spec A$, the corresponding D-module 
$$\Aut_\calE:=\omega_{\Bun_\calG}^{-1/2}\otimes (\calD'\otimes_{A,\varphi_\chi}\bbC)$$ is a Hecke eigensheaf with eigenvalue $\calE$, where $\varphi_\chi:A\to\bbC$ is the homomorphism corresponding to $\chi$, and $\calE$ is the ${^L}G$-local system on $U$ corresponding to the oper $\pi(\chi)\in \on{Op}_{{^L}\frakg}(X)_{\calG}\subset \on{Op}_{{^L}\frakg}(U)$.
\end{cor}
\begin{proof}As explained as above, the isomorphism \eqref{global Hecke} is compatible with the $A$-module structure on both sides. Then $\omega_{\Bun_\calG}^{-1/2}\otimes(\calD'\otimes^L_{A,\varphi}\bbC)$ is a Hecke eigensheaf with eigenvalue $\calE$. Finally $\calD'\otimes^L_{A,\varphi}\bbC=\calD'\otimes_{A,\varphi}\bbC$ by the flatness, giving the corollary.
\end{proof}

\quash{
We need to make use of the following. Assume we have a homomorphism $\calG_1\to \calG_2$ of two models of $G$, which is an isomorphism over $U\subset X$, and for every $x\in X\setminus U$, $(K_1)_x\subset (K_2)_x$ is a closed normal subgroup. Then $K=\prod_x (K_2)_x/(K_1)_x$ acts on $\Bun_{\calG_1}$ and the quotient is $\Bun_{\calG_2}$. 
The D-module $\calD_{\Bun_{\calG_2}}$ is (weakly) $K$-equivariant. 

\begin{lem}
The map $h_\nabla: \Op_{{^L}\frakg}(X)_{\calG_1}\to \Gamma(\Bun_{\calG_1},D')$ factors through $h_\nabla:\Op_{^{L}\frakg}(X)_{\calG_1}\to \Gamma(\Bun_{\calG_1},D')^K$.
\end{lem}
\begin{proof}
Note that the natural map
\[\frakZ_x\to \End \Vac_x\simeq \Vac_x^{K_x}\]
factors through $\frakZ_x\to \Vac_x^{N_{G(F_x)}(K_x)}\subset \Vac_x^{K_x}$.

\[\widehat{U}_{\crit}(\frakg_x)^{\calG(F_x)}\to (U(\frakg_x)/U(\frakg_x)\frakk_x)^{K_x}.\]

We consider the following more in general situation.
Let $K_1\subset K_2$ be a closed normal subgroup and $K_2/K_1$ is an affine algebraic group. Let $S$ be a scheme, equipped with a Harish-Chandra action $(\frakg, K_2)$. Let 
\end{proof}
}

\section{Some geometry of the local Hitchin map}\label{loc Hitchin}

This section is purely local. We fix a point $x\in X$ on the curve. To simplify the notation, we denote the local field $F_x$ by $F$, and its ring of integers by $\calO$, its maximal ideal $\frakm$. Likewise, we use $D$ (resp. $D^\times$) to denote the disc (resp. the punctured disc) around $x$. After choosing a uniformizer $t\in\frakm$, we have $F\simeq \bbC((t))$.
For a line bundle $\calL$ on $D=\Spec\calO$, let $\calL(m)$ denote $\calL(m\cdot x)$ where $x\in D$ is the closed point.

\subsection{The space of invariant polynomials}
Let $\frakg$ be a simple Lie algebra over a field $k$ of characteristic zero, and let $G$ (resp. $G_\ad$ denote the corresponding simply-connected (resp. adjoint) algebraic group. Let 
$$\chi:\frakg^*\to\frakc^*$$ 
denote the usual Chevalley map, where $\frakc^*:=\frakg/\!\! /G=\Spec k[\frakg^*]^G$ is the GIT quotient of $\frakg^*$ under the adjoint action of $G$.
For our purpose, we need to recall some finer structures on $\frakc^*$. 

Let ${^L}\frakg$ be the dual Lie algebra over $k$, and let $\frakc_{{^L}\frakg}= k[{^L}\frakg]^{{^L}G}$ denote its space of characteristic polynomials. Let $\on{Out}({^L}\frakg)$ be the group of outer automorphisms of ${^L}\frakg$. 
Then  the Lie algebra $\frakg$ gives a homomorphism $\psi: \Gal(\overline k/k)\to \on{Out}(\frakg_{\overline k})=\on{Out}({^L}\frakg)$ up to conjugacy. 
Then we have the canonical isomorphism
\[\frakc^*\simeq (\frakc_{{^L}\frakg}\otimes\overline k)^{\Gal(\overline k/k)}.\]
Now we give another description of $\frakc_{{^L}\frakg}$ with the action of $\on{Out}({^L}\frakg)$.

Let $V_{{^L}\frakg}$ be the Lie algebra of the universal unipotent regular centralizer of ${^L}\frakg$, equipped with a  canonical $\bbG_m$-action. 
Let us recall its definition. Namely, let $J_{{^L}\frakg}\to \frakc_{{^L}\frakg}$ denote the universal centralizer group scheme \`{a} la Ng\^{o} \cite[\S~2]{Ng}. Then
$$V_{{^L}\frakg}: =\Lie J_{{^L}\frakg}|_0,$$
where $0\in \frakc_{{^L}\frakg}$ corresponds to the regular nilpotent conjugacy class.
By definition, for a regular nilpotent element $e$, there is a canonical $\bbG_m$-equivariant isomorphism $V_{{^L}\frakg}\simeq {^L}\frakg^e$, and 
after fixing a principal $\fraks\frakl_2$-triple $\{e,h,f\}$ of ${^L}\frakg$, there are the isomorphisms
\begin{equation}\label{flat coordinate}
V_{{^L}\frakg}\simeq f+{^L}\frakg^e\simeq \frakc_{{^L}\frakg}.
\end{equation}
This isomorphism is independent of the choice of the $\fraks\frakl_2$-triple, and therefore endows $\frakc_{{^L}\frakg}$ with a vector space structure.
In addition, if we fix a pinning $({^L}\frakg,{^L}\frakb,{^L}\frakt,e)$, $\Aut({^L}\frakg,{^L}\frakb,{^L}\frakt,e)$ will act on ${^L}\frakg^e$ by linear transformations. By transport of structure, we obtain an action of $\on{Out}({^L}\frakg)$ on $V_{{^L}\frakg}$, which is independent of the choice of the pinning and commutes with the $\bbG_m$-action. In addition, the isomorphism \eqref{flat coordinate} is 
$\on{Out}({^L}\frakg)$-equivariant. If we define a graded $k$-vector space as
\[{^L}V_\frakg:=(V_{{^L}\frakg}\otimes \overline k)^{\Gal(\overline k/k)},\]
we get a canonical $\bbG_m$-equivariant isomorphism
\begin{equation}\label{flat str}
{^L}V_\frakg \simeq \frakc^*,
\end{equation}
giving $\frakc^*$ a graded vector space structure.

\subsection{The local Hitchin map}

Now we apply the above discussion to $k=F$ so $\frakg$ be a simple Lie algebra over $F$. Then $V_{{^L}\frakg}$ has a natural $\bbG_m$-stable $\calO$-structure.
and therefore ${^L}V_{\frakg}$ admits a $\bbG_m$-stable $\calO$-lattice 
$${^L}V_{\frakg,\calO}:=(V_{{^L}\frakg}\otimes \overline{\calO})^{\Gal(\overline{F}/F)}.$$
Similarly there is a canonical $\bbG_m$-stable integral structure on $\frakc^*$, denoted by $\frakc^*_{\calO}$.  The isomorphism \eqref{flat str} is compatible with the integral structures.

Let 
\[{^L}V_{\frakg,\calO}=\bigoplus_i {^L}V_{\frakg,d_i}\]
denote the grading under the action of $\bbG_m$. We arrange the weights appearing in the above decomposition into a non-decreasing sequence $\{d_1,\ldots,d_\ell\}$ such that $d_i$ appears  in the sequence $\dim {^L}(V_{\frakg,d_i}\otimes F)$ times. These positive numbers are sometimes called the (fundamental) degrees of $\frakg$. By \eqref{flat str}, for a $1$-dimensional vector space over $F$ (or an invertible $\calO$-module) $\calL$, 
\[\frakc^*\times^{\bbG_m}\calL^\times\simeq \bigoplus_i \calL^{d_i}\otimes {^L}V_{\frakg,d_i}.\]

Now we regard the Chevalley map $\chi$ as a map of $\bbC$-indschemes.
\quash{We identify $\frakg^*$ with $\frakg\otimes_F \omega_F$ as $k$-indschemes via the residue of the Killing form 
$$\frakg\times (\frakg\otimes\omega_F)\stackrel{B}{\longrightarrow} \omega_F\stackrel{\Res_F}{\longrightarrow} \bbC.$$ } It induces the local Hitchin map of ind-schemes
\[h^{cl}_{x}:\frakg^*\otimes\omega_F\to \frakc^*\times^{\bbG_m}\omega_F^\times=\bigoplus_i \omega_F^{d_i}\otimes {^L}V_{\frakg,d_i}.\]
We will be interested in the image of certain $\calO$-lattices of $\frakg^*$ under the local Hitchin map.

\subsection{Local Hitchin maps for Moy-Prasad subgroups}

Let $P\subset G(F)$ be a parahoric subgroup. It admits a filtration by normal subgroups (the Moy-Prasad filtration)
\[P=P(0)\rhd P(1)\rhd P(2)\rhd\cdots\]
 where $L_P:=P(0)/P(1)$ is the Levi quotient of $P$. Let 
 \[\frakg\supset \frakp=\frakp(0)\supset \frakp(1)\supset\frakp(2)\supset\cdots,\]
denote the corresponding filtration on Lie algebras. Then the exponential map induces isomorphisms
\[\frakp(i)/\frakp(i+1)\simeq P(i)/P(i+1), \quad i\geq 1.\]
We denote $\frakp(1)/\frakp(2)\simeq P(1)/P(2)$ by $V_P$. 

\quash{Recall its definition. Let $x$ be a point in the building of $\frakg$ such that the Lie algebra of the stabilizer of $x$ is $\frakp$. Then one defines $\frakp_{x,r}$ be the Lie algebra generated by $\frakt$ and $\frakg_{\al}$}

There is an integer $m$ such that $\frakm\frakp(i)=\frakp(i+m)$. Therefore, we can extend the above filtration of $\frakp$ to a filtration of $\frakg$ by $\calO$-lattices given by
\[\frakp(i)=\frakm^{-N}\frakp(i+Nm), \quad N\gg 0.\]
Note that for $i<0$, $\frakp(i)$ is not a subalgebra. Let $\frakp(i)^\vee$ be the $\calO$-dual of $\frakp(i)$ in $\frakg^*$, and let $\frakp(i)^\perp=\frakp(i)^\vee\otimes_\calO\omega_\calO$. We have the chain of $\calO$-submodules in $\frakg^*\otimes\omega_F$
\[\cdots\subset \frakp(0)^\perp\subset \frakp(1)^\perp \subset\frakp(2)^\perp\cdots.\]

\begin{prop}\label{size of image}
The following diagram is commutative
\[\begin{CD}
\frakp(n)^\perp @>>>\frakg^*\otimes\omega_F\\
@VVV@VV h^{cl}_{x}V\\
\bigoplus_i \omega_{\calO}^{d_i}((d_i-\lceil\frac{d_i(1-n)}{m}\rceil)) \otimes {^L}V_{\frakg,d_i}@>>>\bigoplus_i \omega_F^{d_i}\otimes {^L}V_{\frakg,d_i}.
\end{CD}\]
\end{prop}
\begin{rmk}
For $n\geq 1$, one shall regard this proposition as a classical limit of the main theorem of \cite{CK}.
\end{rmk}
\begin{proof}
The proof is a simple application of \cite[Theorem 4.1]{RY} in our set-up, which generalizes Kac's work (cf. \cite[Ch. 8]{Ka}) on the realization of (twisted) affine algebras via loop algebras. We will give a relatively self-contained account of Reeder-Yu's result. However, we shall warn the readers from the beginning that our notations are different from Reeder-Yu's.

We choose a pinning $(\frakg,\frakb,\frakt, e)$ of $\frakg$. Let $(\frakg_0,\frakb_0,\frakt_0,e_0)$ be a split form of $\frakg$ together with a pinning defined over $\bbC$. We fix once and for all an isomorphism
\begin{equation}\label{pinning}
(\frakg,\frakb,\frakt,e)\otimes_F\overline{F}\simeq (\frakg_0,\frakb_0,\frakt_0,e_0)\otimes_\bbC\overline{F},
\end{equation}
where $\overline F$ is an algebraic closure of $F$.
It induces a map 
\begin{equation}
\psi:\Gal(\overline F/F)\to \Aut(\frakg_0,\frakb_0,\frakt_0,e_0)\simeq \on{Out}(\frakg_0)
\end{equation} such that the natural action of $\ga\in\Gal(\overline{F}/F)$ on the left hand side corresponds to the action $\psi(\ga)\otimes\ga$ on the right hand side.
These data determine: 
\begin{enumerate}
\item an apartment $\calA$  corresponding to $\frakt$ in the building of $\frakg$;
\item a special point $x_0\in\calA$ corresponding to the parahoric algebra $(\frakg_0\otimes\overline\calO)^{\Gal(\overline{F}/F)}$; 
\quash{\item Using this point, we identify $\calA$ with  the vector space $\xcoch(T_0)^{\Gal(\overline{F}/F)}\otimes\bbR$, where $T_0$ is the torus in $G_0$ whose Lie algebra is $\frakt_0$; 
 An alcove $C\subset\calA$ whose closure contains the point $x_0$ and that is contained in the Weyl chamber in $\xcoch(T_0)\otimes\bbR$ determined by $\frakb_0$.}
\item a set of simple affine roots $\{\al_0,\ldots,\al_\ell\}$ such that $\al_1(x_0)=\cdots=\al_\ell(x_0)=0$ and that the vector parts of $\al_i, i=1,\ldots,\ell$ form the set of simple roots of $\frakg$ with respect to $\frakb$.
\end{enumerate}

Explicitly, we can identify $\calA$ with the $\bbR$-span of the $\Gal(\overline F/F)$-invariant subspace of the cocharacter group $\xcoch(T_{0,\ad})=\Hom_k(\bbG_m,T_{0,\ad})$, where $T_{0,\ad}$ is the adjoint torus with Lie algebra $\frakt_0$. We can also identify $x_0$ with the origin of $\xcoch(T_{0,\ad})^{\Gal(\overline F/F)}$.
If $\frakg=\frakg_0\otimes_{\bbC}F$ and if \eqref{pinning} arises as the base change (i.e. the split case), then $\{\al_1,\ldots,\al_\ell\}$ are the set of simple roots of $\frakg_0$ with respect to $\frakb_0$ (regarded as affine functions on $\calA$), and $\al_0=1-\theta$, where $\theta$ is the highest root. 

Let $\{a_i\}$ be Kac's labels of the affine Dynkin diagram of $\frakg$ (cf. \cite[Ch. 4]{Ka} TABLE Aff 1,2), and $r=1,2,3$ be the order of the image of $\psi:\Gal(\overline F/F)\to \Aut(\frakg_0,\frakb_0,\frakt_0,e_0)$. We regard affine roots as affine functions on $\calA$. Under the (usual) normalization
\[r\sum a_i\al_i=1.\]

Let $C$ be the (closed) alcove in $\calA$ defined by $\{x\in\calA\mid \al_i(x)\geq 1\}$. Clearly, we can prove the proposition just for standard parahorics, i.e. those corresponding to facets of $C$. For such a parahoric $P$, let $x_P$ denote the barycenter of the corresponding facet. We can write the affine coordinates with respect to $\{\al_0,\ldots,\al_\ell\}$  (a.k.a. the Kac coordinates) of $x_P$ as
\[(\frac{s_0}{m},\ldots,\frac{s_\ell}{m}),\]
where $s_i\in\{0,1\}$ and $m=r\sum a_is_i$.

Note that $m(x_P-x_0)$ pairing with any root of $\frakg$ is an integer and therefore defines a cocharacter
\[\eta_P: \bbG_m\to T_\ad,\]
where $T_\ad\subset G_\ad$ is the torus whose Lie algebra is $\frakt$.
The isomorphism \eqref{pinning} induces an isomorphism $\Hom_F(\bbG_m,T_\ad)\simeq \Hom_k(\bbG_m,T_{0,\ad})^{\Gal(\overline F/F)}$, and therefore
 we may regard $\eta_P$ as a cocharacter $\bbG_m\to T_0$.

Let $E/F$ be the (unique) Galois extension of degree $m$ inside $\overline F$. Then $\psi$ factors through $\Gal(E/F)\to\Aut(\frakg_0,\frakb_0,\frakt_0,e_0)$, and the isomorphism \eqref{pinning} descends to $E$. Following Kac, we define a torsion automorphism of $\frakg_0$ as
\[\theta_P: \Gal(E/F)\to\Aut(\frakg_0),\quad \theta_P(\ga)=\psi(\ga) \Ad(\eta_P(\zeta)^{-1}),\]
where $\zeta\in\mu_m$ corresponds to $\ga$ under the canonical identification $\Gal(E/F)=\mu_m$.
We can  write
\[\frakg_0=\bigoplus_{i\in \bbZ/m\bbZ} \frakg_0(i),\]
where $\frakg_0(i)$ is the weight space of $\theta$ corresponding to $i\in \bbZ/m\bbZ=\mu_m^*$. Explicitly, $\frakg_0(i)$ is the $\zeta^{-i}$-eigenspace of $\theta_P(\ga)$.

We endow $\frakg_0\otimes E$ with an action of $\Gal(E/F)$ given by $\theta_P(\ga)\otimes\ga$ for $\ga\in\Gal(E/F)$. Explicitly, let $u$ be a uniformizer of $E$ such that $t=u^m$ is a uniformizer of $F$. Then $\zeta=\ga(u)/u$ so $\ga$ acts on $\frakg_0\otimes u^i$ as $\psi(\ga)\eta_P(\zeta^{-1})\otimes \zeta^i$. 
In particular, for a generator $\ga$ of $\Gal(E/F)$, 
$(\frakg_0\otimes E)^{\theta_P(\ga)\times\ga}$ is the $u$-adic completion of $\bigoplus_{i\in \bbZ} \frakg_0(i)\otimes u^i$.

Note that as $F$-automorphisms of $\frakg_0\otimes E$,
\[\psi(\ga)\times\ga= \Ad(\eta_P(u^{-1}))(\theta_P(\ga)\times \ga) \Ad(\eta_P(u)).\]
Combining with \eqref{pinning}, we obtain the isomorphism constructed by Kac (\cite[Theorem 8.5]{Ka})
\begin{equation}\label{Kac isom}
K:\frakg=(\frakg\otimes_FE)^{\ga}\simeq (\frakg_0\otimes_k E)^{\psi(\ga)\times\ga}\stackrel{\Ad(\eta_P(u))}{\simeq} (\frakg_0\otimes_k E)^{\theta_P(\ga)\times\ga},
\end{equation}
for any choice $\ga$ of the generator of $\Gal(E/F)$. By abuse of notation, we also use $K$ to denote the induced embedding $\frakg\simeq (\frakg_0\otimes_k E)^{\theta_P(\ga)\times\ga}\to \frakg_0\otimes E$.

Let us have a more detailed analysis of the isomorphism $K$. First, 
$$K(\frakm^n\frakt)=\frakt_0^{\theta_P(\ga)}\otimes \frakm^n=\sum_{j\geq n} \frakt_0^{\psi(\ga)}\otimes u^{mj}.$$ 
Next, let $\al$ be an affine root of $\frakg$ and let $\frakg_\al$ denote the corresponding affine root space (which is one-dimensional over $\bbC$). Let $\{b_i\}$ be the set of roots of $\frakg_0$ that restrict to the linear part $\dot{\al}$ of $\al$, and $\{\frakg_{0,b_i}\}$ be the corresponding root spaces. Then
\[K(\frakg_\al)= \sum_i \frakg_{0,b_i} \otimes u^{\langle\eta_P,b_i\rangle+m\al(x_0)}\subset \frakg_0(j)\otimes u^j\]
for $j=\langle\eta_P,b_i\rangle+m\al(x_0)$.
Note that we can also rewrite this number as
\[m\dot{\al}(x_P-x_0)+m\al(x_0)=m\al(x_P).\]
Now recall that by the definition of the Moy-Prasad filtration, $\frakp(n)$ is the subalgebra of $\frakg$ generated by $\frakm^{\left\lfloor \frac{n}{m}\right\rfloor}\frakt$, and all the affine root subspaces $\frakg_{\al}$ such that $\al(x_P)\geq \frac{n}{m}$. Therefore, $K$ restricts to an isomorphism
\[K:\frakp(n)\simeq (\frakg_0\otimes_k \frakm_E^n)^{\theta_P(\ga)\times\ga}=\prod_{j\geq n} \frakg_0(j)\otimes u^j,\]
which is \cite[Theorem 4.1]{RY}

Note that by replacing $\frakg$ by $\frakg^*$ in the above paragraphs, we obtain an isomorphism $K^*: \frakg^*\simeq (\frakg^*_0\otimes E)^{\theta_P(\ga)\times \ga}$
and therefore we obtain
\[DK:\frakg^*\otimes \omega_F\simeq (\frakg^*_0\otimes_k \omega_E)^{\theta_P(\ga)\times\ga}\subset \frakg^*_0\otimes_k \omega_E, \quad \xi(t)\frac{dt}{t}\mapsto mK^*(\xi(t))\frac{du}{u}.\]
We need two properties of $DK$: (i) the following diagram is commutative,
\begin{equation}\label{aux2}
\begin{CD}
\frakg^*\otimes \omega_F@>DK>>\frakg^*_0\otimes \omega_E\\
@VVV@VVV\\
\frakc^*\times^{\bbG_m}\omega^\times_F@>>>\frakc^*_0\times^{\bbG_m}\omega^\times_E.
\end{CD}
\end{equation}
(ii) $K\otimes DK$ commutes with the pairings induced by the residue up to the scalar $m$. I.e.,
for $X\in \frakg, \ \xi\frac{dt}{t}\in \frakg\otimes\omega_F$,
\begin{equation}\label{inv res}
\Res_F (\xi,X)\frac{dt}{t}= m\Res_E (K^*(\xi),K(X))\frac{du}{u}=m\Res_E(DK(\xi\frac{du}{u}),K(X)).
\end{equation}

Since under the residue pairings, the orthogonal complement of $(\frakg_0\otimes \frakm_E^n)$ is $(\frakg^*_0\otimes \frakm_E^{-n}\omega_{\calO_E})$, and the orthogonal complement of $\frakp(n)$ is $\frakp(n)^\perp$,
we obtain from \eqref{inv res} that 
\begin{equation}\label{aux1}
\frakp(n)^\perp=DK^{-1}(\frakg^*\otimes\omega_F\cap \frakm_E^{-n}(\frakg^*_0\otimes\omega_{\calO_E})).
\end{equation}

Note that under the map $\frakg_0\otimes\omega_E\to \frakc_0\otimes \omega_E$, the closed subscheme
$\frakm_E^{-n}(\frakg^*_0\otimes\omega_{\calO_E})$ maps surjectively to $\bigoplus \omega_{\calO_E}^{d_i}(nd_i)$. This is clear for $n=0$ and therefore for all $n$.
On the other hand,
\begin{lem} Inside $\omega^d_{E}$,
\[ \omega_{\calO_E}^{d}(n)\cap \omega^d_F=  \omega_{\calO_F}^{d}(\lceil\frac{(d-n)}{m}\rceil).\]
\end{lem}
\begin{proof}Let $f(t)(\frac{dt}{t})^d\in \omega_F^d$, then $f(u^m)m^d(\frac{du}{u})^d\in \omega_{\calO_E}^d(n)$ if and only if 
$$m\on{ord}_t(f)-d\geq -n.$$ So $\on{ord}_t(f)\geq \lceil\frac{(d-n)}{m}\rceil$.
\end{proof}

Now the proposition follows easily from \eqref{aux2} \eqref{aux1} and the above lemma.
\end{proof}

\begin{rmk}\label{dual}
When studying harmonic analysis on $\frakg$, it is customary to regard the (local) Hitchin map as a map $\frakg\otimes \omega_F\to \frakc\times^{\bbG_m}\omega_F$, and it's natural to describe the image of $\frakp(n)$ under this map. But this reduces the study the image of $\frakp(n)^\perp$ under $h_x^{cl}$ by the following reason.
Note that the isomorphism $\frakg^*\simeq \frakg$ induced by the Killing form intertwines $K$ and $K^*$. Therefore,
\begin{equation*}
\frakp(n)^\perp= \frakp(1-n)\otimes_{\calO}\omega_\calO(1)=\frakp(1-n-m)\otimes_{\calO}\omega_{\calO}.
\end{equation*}
\end{rmk}

\medskip

We consider a few well-known special cases of Proposition \ref{size of image}.
If $n=0$, it says that the local Hitchin map restricts to a map
\[\frakp(0)^\perp\to \bigoplus_{i}\omega^{d_i}_{\calO}(\left\lfloor\frac{(m-1)d_i}{m} \right\rfloor)\otimes {^L}V_{\frakg,d_i}\]
In particular, 
\begin{itemize}
\item if $m=1$, so $\frakp\simeq \frakg_0\otimes_k \calO$ is hyperspecial, then we have the unramified local Hitchin map 
$$\frakp(0)^\perp\to \bigoplus_{i}\omega_{\calO}^{d_i}\otimes {^L}V_{\frakg,d_i}=\frakc^*_{\calO}\times^{\bbG_m}\omega^\times_{\calO}.$$

\item If $m$ is the (twisted) Coxeter number of $\frakg$, then $\frakp$ must be an Iwahori subalgebra and the local Hitchin map restricts to the map
$$\frakp(0)^\perp\to \bigoplus_{i}\omega_{\calO}^{d_i}(d_i-1)\otimes {^L}V_{\frakg,d_i}.$$
\quash{
\item If $m=2$, then $\frakp$ is the so-called Gross parahoric and we have
$$\frakp(0)^\perp \to \bigoplus_{i}\omega_{\calO}^{d_i}(\left\lfloor\frac{d_i}{2}\right\rfloor)\otimes V_{\frakg,d_i}.$$}
\end{itemize}
It's known that in both cases the maps are surjective.

If $n=1$, it says that the local Hitchin map restricts to a map
\begin{equation}\label{image p(1)}
\frakp(1)^\perp\to\bigoplus_{i}\omega^{d_i}_{\calO}(d_i)\otimes V_{\frakg,d_i}\simeq \frakc_{\calO}\times^{\bbG_m}\omega_{\calO}(1)^\times=: \on{Hitch}(D)_{\on{RS}},
\end{equation}
which is known to be surjective if $\frakp$ is an Iwahori subalgebra (so $\frakp(1)=[\frakp(0),\frakp(0)]$). Since for a general parahoric subalgebra $\frakp$, $\frakp(1)^\perp$ contains the orthogonal complement of the pro-nilpotent radical of some Iwahori subalgebra, we see that
\begin{cor}
The map \eqref{image p(1)} is surjective for every $\frakp$.
\end{cor}

\medskip

Now we consider the special case when $\theta_P:\mu_m\to\Aut(\frakg_0)$ corresponding to the parahoric $P$ constructed above is principal (or called $N$-regular as in \cite{P}). This means that $\theta_P$ is
conjugated by an inner automorphism to
\[\theta'_m=\psi \rho: \mu_m\to \Aut(\frakg_0),\]
where $\rho$ as usual denotes the half sum of positive roots of $\frakg_0$ (with respect to $\frakb_0$). In this case, there exists a regular nilpotent element $e'$ of $\frakg_0$ (maybe different from $e_0$ we fixed before) in $\frakg_0(1)$.

\quash{According to Kac (\cite[Theorem 8.6]{Ka}), $\theta_P$ and $\theta'$ are conjugate if and only if the Kac coordinate of $x_P$ can be transferred to $(\frac{r+1-h_{\psi}}{m},\frac{1}{m},\ldots,\frac{1}{m})$ under the Iwahori-Weyl group of $G(F)$. }
Clearly, $m=1$ (so that $\frakp$ is hyperspecial) is the case. In addition, it is proved in \cite[Corollary 5.1]{RY} that if $m$ is a regular elliptic number of $\frakg_0$, then $\theta_P$ is principle. In particular, if $P=I$ is an Iwahori subgroup, then $\theta_I$ is principal.

Assume that $\theta_P$ is principal. Let us denote 
$$\on{Hitch}(D)_{\frakp(n)}=\bigoplus_i\omega_{\calO}^{d_i}(d_i-\lceil\frac{d_i(1-n)}{m}\rceil)\otimes {^L}V_{\frakg,d_i}\subset \on{Hitch}(D^\times):=\frakc^*\times^{\bbG_m}\omega_F^\times.$$
\begin{prop}\label{surj}
Assume that $\theta_P$ is principle. Then the map 
$$\frakp(2)^\perp\to\on{Hitch}(D)_{\frakp(2)}$$ is surjective. In particular, the map $\on{Fun}\on{Hitch}(D)_{\frakp(2)}\to \on{Fun}\frakp(2)^\perp$ is injective.
\end{prop}
\begin{rmk}There are only a few parahorics $\frakp$ of $\frakg$ such that $\theta_P$ is principal. It will be interesting to describe the precise image of $\frakp(n)^\perp$ under the local Hitchin map for a general parahoric $\frakp$ and general $n$. We are informed that D. Baraglia and M. Kamgarpour have made some progress in this direction. 
\end{rmk}
\begin{proof}
It is convenient to identify $\frakg_0^*$ with $\frakg_0$ using the Killing form. We can choose a principal $\fraks\frakl_2$-triple $\{e',h',f'\}$ of $\frakg_0$ where $e'\in\frakg_0(1), h'\in \frakg_0(0), f'\in \frakg_0(-1)$.
Then the Kostant section 
\[\kappa:\frakc_0\times^{\bbG_m}\omega^\times_E\to \frakg_0\otimes \omega_E\]
associated to $\{e',h',f'\}$ maps the element
$$(t^{-\lfloor\frac{d_1}{m}\rfloor}c_1(t)(\frac{dt}{t})^{d_1}\otimes v_1,\ldots,t^{-\lfloor\frac{d_\ell}{m}\rfloor}c_\ell(t)(\frac{dt}{t})^{d_\ell}\otimes v_\ell)\in \on{Hitch}(D)_{\frakp(2)}$$ to
\begin{equation}\label{section} 
(f'+ \sum_i u^{-d_i-m\lfloor\frac{d_i}{m}\rfloor}c_i(u^m)p_i) du,
\end{equation}
where $p_i'\in \frakg_0^{e'}$ such that $\chi(f'+p_i)=v_i$. Let us regard $h'$ as a cocharacter $h':\bbG_m\to G_0$. Then after conjugating by $\Ad(h'(u))$,  the element \eqref{section} becomes
\[(u^{-1}f'+  \sum_i u^{d_i-m\lfloor\frac{d_i}{m}\rfloor-1}c_i(u^m)p_i)\frac{du}{u}.\]
As $p_i\in \frakg_0(d_i-1)$, the pre-image under $DK$ of the above element belongs to $\frakp(2)^\perp$. This implies the proposition.
\end{proof}

There is the following map
\begin{equation}\label{loc res}
\on{Hitch}(D)_{\frakp(2)}\to \bigoplus_{d_i\mid m} \omega^{d_i}_\calO(d_i+\frac{d_i}{m})/\omega^{d_i}_\calO(d_i+\frac{d_i}{m}-1)\otimes {^L}V_{\frakg,d_i},
\end{equation}
which fits into the commutative diagram
\[\begin{CD}
\frakp(2)^\perp @>>> \frakg^*_0(-1)\otimes\omega_{\calO_E}(2)|_0=\frakp(2)^\perp/\frakp(1)^\perp\\
@VVV@VVV\\
\on{Hitch}(D)_{\frakp(2)}@>>>  \bigoplus_{d_i\mid m} \omega^{d_i}_\calO(d_i+\frac{d_i}{m})/\omega^{d_i}_\calO(d_i+\frac{d_i}{m}-1)\otimes {^L}V_{\frakg,d_i}.
\end{CD}\]
In addition, the right vertical map factors through
\begin{equation}\label{vinberg quot}
\frakg^*_0(-1)\otimes \omega_{\calO_E}(2)|_0\to (\frakg^*_0(-1)/\!\!/L_P)\times^{\bbG_m}\omega_{\calO_E}(2)|_0\to \bigoplus_{d_i\mid m} \omega^{d_i}_\calO(d_i+\frac{d_i}{m})/\omega^{d_i}_\calO(d_i+\frac{d_i}{m}-1)\otimes {^L}V_{\frakg,d_i}.
\end{equation}
Since $\theta_P$ is principal, the map 
$$(f'+\frakg_0^{e'})\cap \frakg_0(-1)=\{f+ \sum_{m\mid d_i} c_ip_i\}\to \frakg_0(-1)/\!\!/L_P$$ is an isomorphism by \cite[Theorem 3.5]{P}. Therefore the last map in \eqref{vinberg quot} is an isomorphism. 

Note that the pairing $\Res_F: \frakp(2)^\perp\times \frakp(1)\to\bbC$ identifies $\frakg^*_0(-1)\otimes\omega_{\calO_E}(2)|_0$ with the dual $V_P^*$ of $\frakg_0(1)=V_P$, and therefore we naturally 
identify $(\frakg^*_0(-1)/\!\!/L_P)\times^{\bbG_m}\omega_{\calO_E}(2)|_0$ with $V^*_P/\!\!/L_P$.
We therefore obtain

\begin{prop}\label{new residue}
Assume that $\theta_P$ is principal. Then
there is the following commutative diagram
\[\begin{CD}
\frakp(2)^\perp @>>> V_P^*\\
@VVV@VVV\\
\on{Hitch}(D)_{\frakp(2)}@>>> V_P^*/\!\!/L_P,
\end{CD}\]
with all arrows surjective.
\end{prop}

\begin{rmk}\label{first residue}
We will not make use of this remark. One should compare the above proposition with the following commutative diagram, which holds for any $\frakp$, 
\[\begin{CD}
\frakp(1)^\perp @>>> ( \Lie L_P)^*\simeq \frakp(1)^\perp/\frakp(0)^\perp\\
@VVV@VVV\\
\on{Hitch}(D)_{\on{RS}}@>>> (\frakc^*_{\calO}\times^{\bbG_m}\omega_{\calO}(1)^{\times})|_0.
\end{CD}\]
Note that the bottom arrow in the diagram is the classical limit of the residue map studied in \cite{BD,FGa}. 
\end{rmk}

\section{Endomorphisms of some vacuum modules}\label{endo alg}
We continue to fix a point $x\in X$, and for simplicity write $F_x$ by $F$, etc. 

\subsection{Endomorphisms of some vacuum modules}Let us quantize the picture discussed in the previous section.
As in \S~\ref{vac mod}, we restrict to the case when $\frakg$ is a split Lie algebra. Without loss of generality we can assume that $\frakg=\frakg_0\otimes_kF$.
Also we denote by 
 $\hat{\frakg}_c$ the \emph{critical} central extension of $\frakg$ as in \S~\ref{vac mod}, and let $\hat{U}_c(\frakg)$ be the corresponding completed universal enveloping algebra.

Let $P\subset G(F)$ be a parahoric subgroup and $P=P(0)\rhd P(1)\rhd P(2)\rhd\cdots$ be the corresponding Moy-Prasad filtration. As $P(n)$ for $n>0$ is a pro-unipotent group, we have a natural splitting $\widehat{\frakp(n)}=\frakp(n)\oplus \bbC\mathbf{1}$, where $\widehat{\frakp(n)}$ denotes the preimage of $\frakp(n)$ in $\hat{\frakg}_c$.

Now for $n>0$, we define
\begin{equation}\label{vacp}
\Vac_{\frakp(n)}= \on{Ind}_{\frakp(n)\oplus \bbC\mathbf{1}}^{\hat{\frakg}_c}(\bbC),
\end{equation}
where $\frakp(n)$ acts trivially on $\bbC$ and $\mathbf{1}$ acts as the identity.

Let $\frakZ$ be the center of $\hat{U}_c(\frakg)$ endowed with the natural filtration. By \cite[Theorem 3.7.8]{BD}, the natural map
\begin{equation}\label{classical center}
\on{gr}\frakZ\to \on{Fun}\on{Hitch}(D^\times)
\end{equation}
is an isomorphism.

We write
\[\frakZ\to \frakZ_{\frakp(n)}\hookrightarrow \End \Vac_{\frakp(n)}\simeq \Vac_{\frakp(n)}^{P(n)}\subset \Vac_{\frakp(n)}.\]
All the maps are strictly compatible with the filtrations so that taking the associated graded gives
\[\on{Fun}\on{Hitch}(D^\times)\twoheadrightarrow \on{Fun}\on{Hitch}(D)_{\frakp(i)}\subset \on{Fun}(\frakp(i)^\perp).\]

As
\[\Vac_{\frakp(2)}=\on{Ind}^{\hat{\frakg}_{c}}_{\frakp(2)+\bbC\mathbf{1}}\bbC=\on{Ind}^{\hat{\frakg}_{c}}_{\frakp(1)+\bbC\mathbf{1}}U(V_P),\]
we obtain an injective algebra homomorphism
\[U(V_P)\simeq \End_{\frakp(1)}U(V_P)\subset \End(\Vac_{\frakp(2)}).\]

\begin{prop}\label{endo diag}
Assume that $\theta_P$ is principal. Then there is the following natural commutative diagram
\[\begin{CD}
U(V_P)^{L_P}@>>> \frakZ_{\frakp(2)}\\
@VVV@VVV\\
U(V_P)@>>> \End(\Vac_{\frakp(2)})
\end{CD}\]
In fact, $\frakZ_{\frakp(2)}\cap U(V_P)=U(V_P)^{L_P}$.
\end{prop}

\begin{proof}Let $W$ be the intersection of $\frakZ_{\frakp(2)}$ and $U(V_P)$ inside $\End(\Vac_{\frakp(2)})$.  Under the following natural maps, 
\[\End(\Vac_{\frakp(2)})\simeq \Vac_{\frakp(2)}^{\frakp(2)}\subset \Vac_{\frakp(2)},\]
the subalgebra $\frakZ_{\frakp(2)}\subset \End(\Vac_{\frakp(2)})$ is in fact contained in $\Vac_{\frakp(2)}^P\subset \Vac_{\frakp(2)}$, and $U(V_P)$ is just the subspace in $\Vac_{\frakp(2)}$ generated from the vacuum vector by $\frakp(1)$.
Therefore $W$ is contained in $U(V_P)^{L_P}=U(V_P)\cap \Vac_{\frakp(2)}^P$.
The filtrations on all the above spaces are induced from the natural filtration on $\hat{U}_{c}(\frakg)$. Therefore $\on{gr}W$ is the intersection of $\on{gr}\frakZ_{\frakp(2)}=\on{Fun}\on{Hitch}(D)_{\frakp(2)}$ and $\on{gr}U(V_P)=\on{Fun}(V_P^*)$ inside $\on{gr}\Vac_{\frakp(2)}=\on{Fun} \frakp(2)^\perp$. Therefore by Proposition \ref{new residue}, we have inclusions
\[(\on{Fun} V_P^*)^{L_P}\subset\on{gr}W\subset \on{gr}(U(V_P)^{L_P})\subset (\on{Fun} V_P^*)^{L_P}.\]
Clearly, the composition of these inclusions is just the identity map.
Therefore $\on{gr}W=\on{gr}(U(V_P)^{L_P})$. Together with $W\subset U(V_P)^{L_P}$, it implies that $W=U(V_P)^{L_P}$.
\end{proof}

\begin{rmk}
(i) In the case when $P$ is an Iwahori subgroup, T.-H. Chen (private communication) proved this proposition by a different (and more direct) argument. He then in turn deduced Proposition \ref{new residue} from it  in this case.

(ii) It is easy to see that the subalgebras $U(V_P)$ and $\frakZ_{\frakp(2)}$ in $\End(\Vac_{\frakp(2)})$ commute with each other, and therefore induces a map
\[U(V_P)\otimes_{U(V_P)^{L_P}} \frakZ_{\frakp(2)}\to \End(\Vac_{\frakp(2)}).\]
This is probably an isomorphism.
\end{rmk}

\subsection{Endomorphism algebras and opers}
Recall that we denote by ${^L}\frakg$ the Langlands dual Lie algebra of $\frakg$, equipped with a Borel subalgebra ${^L}\frakb$. As we assume that $\frakg$ is split, ${^L}V_\frakg=V_{{^L}\frakg}$ and we will following \cite{BD} to use the later notation in this subsection.
Recall that there is a natural $V_{{^L}\frakg}\otimes\omega_F$-torsor structure on $\on{Op}_{{^L}\frakg}(D^\times)$ (cf. \cite[3.1.9]{BD}), which
 induces a natural filtration on $A_{{^L}\frakg}(D^\times)$  (cf. \cite[\S~3.1.13]{BD}) with a canonical isomorphism
$$\on{gr}A_{{^L}\frakg}(D^\times)\simeq \on{Fun} V_{{^L}\frakg}\otimes\omega_F\simeq \on{Fun}\frakc_{{^L}\frakg}\times^{\bbG_m}\omega_F^\times\simeq \on{Fun}\frakc^*\times^{\bbG_m}\omega_F^\times=\on{Fun}\on{Hitch}(D^\times).$$
One of the properties of the Feigin-Frenkel isomorphism \eqref{FeFr} is that it is
compatible with the filtration such that the associated graded gives \eqref{classical center}. Then we may regard $\Spec \frakZ_{\frakp(n)}=: \on{Op}_{{^L}\frakg}(D)_{\frakp(n)}$ as a closed subscheme of $\on{Op}_{{^L}\frakg}(D^\times)$, which we now describe for $n=2$.

Let $\Op_{^{L}\frakg}(D)$ be the scheme (of infinite type) of
${^L}\frakg$-opers on $D$. As before, it has a natural $V_{{^L}\frakg}\otimes\omega_{\calO}$-torsor structure. We have the natural inclusions of pro-unipotent groups
\[V_{{^L}\frakg}\otimes\omega_{\calO}\subset\bigoplus \omega_{\calO}^{d_i}(d_i+\left\lfloor\frac{d_i}{m}\right\rfloor)\otimes V_{{^L}\frakg,d_i}.\]
Then 
\begin{lem} The scheme
$\on{Op}_{{^L}\frakg}(D)_{\frakp(2)}$ is the $(\bigoplus_i \omega_{\calO}^{d_i}(d_i+\left\lfloor\frac{d_i}{m}\right\rfloor)\otimes V_{{^L}\frakg,d_i})$-torsor induced from the $V_{{^L}\frakg}\otimes\omega_{\calO}$-torsor $\Op_{^{L}\frakg}(D)$.
\end{lem}
\begin{proof}
Let us temporarily denote by $\on{Op}_{{^L}\frakg}(D)'_{\frakp(2)}$ this induced torsor. Since $\on{Op}_{{^L}\frakg}(D^\times)$ is also induced from $\on{Op}_{{^L}\frakg}(D)$ via $V_{{^L}\frakg}\otimes\omega_{\calO}\subset V_{{^L}\frakg}\otimes\omega_F$, $\on{Op}_{{^L}\frakg}(D)'_{\frakp(2)}$ is a closed subscheme of $\on{Op}_{{^L}\frakg}(D^\times)$ and the filtration on $\on{Fun}\on{Op}_{{^L}\frakg}(D)'_{\frakp(2)}$ is the quotient filtration on $A_{{^L}\frakg}(D^\times)$. Therefore, it is enough to show that the associated graded of $\on{Fun}\on{Op}_{{^L}\frakg}(D)'_{\frakp(2)}$ and of $\on{Fun}\on{Op}_{{^L}\frakg}(D)_{\frakp(2)}$ coincide in $\on{Hitch}(D^\times)$, which is clear.
\end{proof}

More explicitly, we can describe $\on{Op}_{{^L}\frakg}(D)_{\frakp(2)}$ as follows. We fix a uniformizer $t\in F$, and a principal $\fraks\frakl_2$-triple $\{e,h,f\}$ of ${^L}\frakg$ with $e\in{^L}\frakb$, then it is well-known (e.g. \cite[\S~3.5.6]{BD} or \cite[\S~4.2.4]{F}) that $\on{Op}_{{^L}\frakg}(D^\times)$ can be identified with
\begin{equation}\label{oper space}
\{\nabla=d+(f+u)dt \mid u\in {^L}\frakb(F)\}/{^L}U(F)\simeq \{\nabla=d+(f+v)dt\mid v\in V_{{^L}\frakg}\otimes F\},
\end{equation}
where ${^L}U$ is the unipotent radical of the ${^L}B$.
Then $\on{Op}_{{^L}\frakg}(D)_{\frakp(2)}$ is the subspace of operators as above such that 
\begin{equation}\label{irr oper}
v\in \bigoplus_i \omega^{d_i}_{\calO}(d_i+\left\lfloor \frac{d_i}{m}\right\rfloor)\otimes V_{{^L}\frakg,d_i}.
\end{equation}
\begin{rmk}
There is the notion of the slope of a ${^L}\frakg$-local system on the punctured disc $D^\times$ (e.g. \cite[\S~5]{FG} or \cite[\S~2]{CK}). It is not hard to see that $\on{Op}_{{^L}\frakg}(D)_{\frakp(2)}$ is the (reduced) subscheme of $\on{Op}_{{^L}\frakg}(D^\times)$ such that the underlying local system has slope $\leq \frac{1}{m}$. We do not use this fact in the sequel.
\end{rmk}

\section{Proof of a conjecture in \cite{HNY}}\label{Proof} Now we specialize the group scheme $\calG$ over $X=\bbP^1$. Let $G_0$ be a
simple, simply-connected complex Lie group, of rank $\ell$. Let us
fix $B_0\subset G_0$ a Borel subgroup and $B_0^{\opp}$ an opposite Borel
subgroup. The unipotent radical of $B_0$ (resp. $B_0^{\opp}$) is denoted by
$U_0$ (resp. $U_0^{\opp}$). Following \cite{HNY}, we denote by $\calG(0,1)$
the group scheme on $\bbP^1$ obtained from the dilatation of the constant group scheme
$G_0\times\bbP^1$ along $B_0^{\opp}\times\{0\}\subset G_0\times\{0\}$ and along
$U_0\times\{\infty\}\subset G_0\times\{\infty\}$. Explicitly, 
$$\calG(0,1)(\calO_0)=\on{ev}_0^{-1}(B_0^{\opp}),\quad \calG(0,1)(\calO_\infty)=\on{ev}_\infty^{-1}(U_0),$$ where $\on{ev}_0: G_0(\calO_0)\to G_0$ (resp. $\on{ev}_\infty:G_0(\calO_\infty)\to G_0$) is the evaluation map. 
Following \emph{loc.
cit.}, we denote $I(1)=\calG(0,1)(\calO_\infty)$.

Let $\calG(0,2)\to\calG(0,1)$ be the dilatation of $\calG(0,1)$ at $\infty$, such
that it is an isomorphism away from $\infty$, and that at $\infty$ it induces
$$\calG(0,2)(\calO_\infty)=I(2):=[I(1),I(1)]\subset I(1)=\calG(0,1)(\calO_\infty).$$ 
To simplify notations, in the sequel, $\calG(0,1)$ is denoted by $\calG'$ and $\calG(0,2)$ is denoted by $\calG$.

Note that we can apply all the results in \S~\ref{loc Hitchin} to the standard Iwahori subgroup $I=I(0)\subset G(F_\infty)$. Following notations in that section, let
$V=V_I=I(1)/I(2)$, which is isomorphic to
$\prod_{i=0}^{\ell}U_{\al_i}$, where $\al_i$ are simple affine roots,
and $U_{\al_i}$ are the corresponding affine root groups. We identify $V$ with its Lie algebra via the exponential map as explained before. Let $T=L_I= I(0)/I(1)$, which is the Cartan torus acting on $V$ by conjugation.
\quash{
Let us choose for
each $\al_i$ an isomorphism $\Psi_i:U_{\al_i}\simeq\bbG_a$. Then we
obtain a well-defined morphism
\[\Psi: I(1)\to I(1)/I(2)\simeq \prod_{i=0}^{\ell}U_{\al_i}\simeq\prod\bbG_a\stackrel{\on{sum}}{\to}\bbG_a.\]
Let $I_\Psi:=\ker\Psi\subset I(1)$.}

As explained in \emph{loc. cit.}, there is an open substack $*\subset \Bun_{\calG'}$ corresponding to trivializable $\calG'$-torsors. Its pre-image in
$\Bun_{\calG}$ is isomorphic to $V$, denoted by $\mathring{\Bun}_\calG$.

\begin{lem}\label{goodness}
The stack $\Bun_{\calG}$ is good in the sense of \cite[\S~1.1.1]{BD}.
\end{lem}
\begin{proof}Since $\Bun_{\calG}$ is a principal bundle
over $\Bun_{\calG'}$ under the group $V$,
it is enough to show that $\Bun_{\calG'}$ is good. But this follows from the fact that $\Bun_{\calG'}$ has a stratification by
elements in the affine Weyl group of $G$ such that the stratum
corresponding to $w$ has codimension $\ell(w)$ and the stabilizer
group has dimension $\ell(w)$. 

Indeed, by \cite[Proposition 1]{HNY},
\[\Bun_{\calG'}=\calG'_{out}\backslash \Gr_{\calG',0},\]
where $\calG'_{out}=\calG'(\bbP^1\setminus\{0\})$ is the ind-group representing sections of $\calG'$ over $\bbP^1\setminus\{0\}$.
By \cite[Corollary 3]{HNY}, $\dim\Bun_{\calG'}=0$.
Note that $\Gr_{\calG',0}$ is just the usual affine flag variety for $G_0(F_0)$. By \cite[Theorem 7]{Fa}, for every $w$ in the affine Weyl group $W_{\on{aff}}$, the double quotient
\[C_w:=\calG'_{out}\backslash \calG'_{out}\ \cdot\ w\ \cdot\ \calG'(\calO_0)/\calG'(\calO_0)\]
is represented by an Artin stack, locally closed in $\Bun_{\calG'}$. In addition, $\dim C_w=-\ell(w)$. This proves the claimed fact.
\end{proof}

\begin{rmk}\label{sqr5}
Again, according to \cite[Theorem 7]{Fa}, the closure $\overline{C}_s$ of $C_s$ for  a simple reflection $s$ is a divisor on $\Bun_{\calG'}$. Together with the calculation in \cite[\S~4.1]{Z}, we see twice of the sum of all these divisors gives the line bundle $\omega_{\Bun_{\calG'}}^{-1}$. In particular, the square root $\omega_{\Bun_{\calG'}}^{1/2}$ exists. Note that the pullback of $\omega_{\Bun_{\calG'}}^{1/2}$ is isomorphic to $\omega_{\Bun_{\calG}}^{1/2}$, and therefore $\omega_{\Bun_{\calG}}^{1/2}$ is canonically trivialized over $\mathring{\Bun}_\calG$ .  
\end{rmk}

\quash{Let $S_w$ denote the preimage in $\Bun_{\calG(0,\Psi)}$ of
$C_w\subset \Bun_{\calG(0,1)}$. Then
$S_1\simeq\bbA^1$, and 
\begin{lem}
For a simple reflection $s$, $S_1\cup
S_s\simeq\bbP^1$. In particular, any regular function on
$\Bun_{\calG(0,\Psi)}$ is constant.
\end{lem}
\begin{proof}
Indeed, let $i_s:\SL_2\to G(F_0)$ denote the $\SL_2$ for the simple affine root corresponding to $s$. It induces an embedding $\bbP^1\to \Gr_{\calG(0,1),0}$, still denoted by $i_s$. As $\bbP^1\times \calG(0,1)_{out}\to \Gr_{\calG(0,1),0}$ is an open immersion, we see it projects to an open immersion $\bbP^1\to \Bun_{\calG(0,1)}$.
\end{proof}}

Now we consider the Hitchin map
\[h^{cl}: T^*\Bun_\calG\to \on{Hitch}(X)_{\calG}\subset \on{Hitch}(U).\]
Applying results in \S~\ref{loc Hitchin} and the similar argument of Lemma \ref{support}, it is easy to calculate $\on{Hitch}(X)_{\calG}$ in this case
 
\[\bigoplus_ {i<\ell} \Gamma(X,\omega_X^{d_i}((d_i-1)\cdot 0+ d_i\cdot \infty)\otimes V_{\frakg_0,d_i}\bigoplus \Gamma(X,\omega_X^{d_\ell}((d_\ell-1)\cdot 0+ (d_\ell+1)\cdot \infty)\otimes V_{\frakg_0,d_\ell},\]
which is isomorphic to $\bbA^1$.

Let $\mu: T^*\Bun_{\calG}\to V^*$ be the moment map for the action of $V$ on $\Bun_\calG$. 
\begin{lem}\label{flat moment}
The moment map $\mu: T^*\Bun_\calG\to V^*$ is flat.
\end{lem}
\begin{proof}Consider the Hamiltonian reduction $\mu^{-1}(0)/V\simeq T^*\Bun_{\calG'}$. By the proof of Lemma \ref{goodness}, $\Bun_{\calG'}$ is good, so $T^*\Bun_{\calG'}$ is of dimension zero. This implies that $\dim \mu^{-1}(0)=\dim V$. Therefore, $\mu$ is flat.
\end{proof}

We also need a global version of
Proposition \ref{new residue} in this case.
\begin{lem}\label{global new residue}
There is the following commutative diagram, with all arrows surjective.
\[\begin{CD}
T^*\Bun_\calG @>\mu>> V^*\\
@VVV@VVV\\
\on{Hitch}(X)_{\calG}@>>> V^*/\!\! /T.
\end{CD}\]
In addition, the bottom arrow is an isomorphism. 
\end{lem}
\begin{proof}
By restricting to the formal neighbourhood of $\infty$, the global Hitchin map embeds into the local Hitchin map, and therefore Proposition \ref{new residue} gives the commutative diagram. To see the bottom arrow is an isomorphism, it is enough to observe that in our special case the composition
$$\on{Hitch}(X)_\calG\to \on{Hitch}(D)_{\frakp(2)}\to \bigoplus_{d_i\mid m} \omega^{d_i}_\calO(d_i+\frac{d_i}{m})/\omega^{d_i}_\calO(d_i+\frac{d_i}{m}-1)\otimes{^L}V_{\frakg,d_i}= \omega^{d_\ell}_\calO(d_\ell+1)/\omega^{d_\ell}_\calO(d_\ell)$$ is an isomorphism.
\end{proof}

\begin{rmk}\label{general vinberg}
The above discussions can be generalized. Namely, given any simple, simply-connected group $G$ over $\bbC((t))$ and a parahoric subgroup $P\subset G(\bbC((t)))$, one can construct the corresponding group scheme $\calG=\calG(0,2)_P$ over $\bbP^1$, unramified over $\bbP^1\setminus\{0,\infty\}$, and $\calG(\calO_0)\simeq P^{\on{opp}}$ and $\calG(\calO_\infty)= P(2)$, generalizing $\calG(0,2)$. Then Lemma \ref{goodness} and \ref{flat moment} generalize in this case. In addition, if $\theta_P$ is principal, Lemma \ref{global new residue} generalizes as well, i.e. we have
\[\begin{CD}
T^*\Bun_\calG @>\mu>> V_P^*\\
@VVV@VVV\\
\on{Hitch}(X)_{\calG}@>>> V_P^*/\!\! /L_P.
\end{CD}\]
To prove this, applying Proposition \ref{size of image}, we conclude that $\on{Hitch}(X)_\calG\to V^*_P/\!\!/L_P$ is a closed embedding. But $T^*\Bun_\calG\to V^*_P\to V^*_P/\!\!/L_P$ is clearly surjective, so $\on{Hitch}(X)_\calG\simeq V^*_P/\!\!/L_P$ if $\theta_P$ is principal.
\end{rmk}

Now we quantize the above picture.
Let us describe $\on{Op}_{{^L}\frakg}(X)_{\calG}$ in this
case. 

At $0\in\bbP^1$, $K_0=\calG(\calO_0)=I^{\opp}:=\on{ev}_0^{-1}(B^{\opp})$. We use the inclusion $K_0\subset G_0(\calO_0)$ to split $\hat{\frakk}_0\simeq \frakk_0\oplus\bbC\mathbf{1}$ (see Remark \ref{diff splitting} for the subtlety). Then Lemma \ref{vacuum I} specializes to
\[\on{Vac}_0=\Ind_{\Lie I^{\opp} + \bbC \bf{1}}^{\hat{\frakg}_{c,0}}(\bbC_{\rho}),\]
which is just the Verma module $\bbM^{\opp}_{\rho}$ for $I^{\opp}$. 
Here $\rho$ is half sum of roots of $B$, so is \emph{anti-dominant} w.r.t. $B^{\opp}$. It is known
(\cite[Corollary 13.3.2]{FGa} or \cite[Theorem 9.5.3]{F}) that
$\on{Fun}\on{Op}_{{^L}\frakg}(D^\times_0)\to\End (\bbM^{\opp}_{\rho})$
induces an isomorphism
\[\on{Fun}\on{Op}_{{^L}\frakg}(D_0)_{\varpi(0)}\simeq\End(\bbM^{\opp}_{\rho}),\]
where $\on{Op}_{{^L}\frakg}(D_0)_{\varpi(0)}$ is the scheme of
${^L}\frakg$ opers on $D_0$ with regular singularities and
residue $\varpi(0)$. In fact, in \emph{loc. cit.}, this is proved for the Verma module $\bbM_{-\rho}$ of $I$. But
if we choose an element $\tilde{w}_0\in G(\bbC)$, conjugation by which switches $B$ and $B^{\opp}$, and define the new action of $\hat{\frakg}_{c,0}$ on $\bbM_{-\rho}$ by $X\cdot v=\Ad_{\tilde{w}_0}(X) v$, then the resulting module
$\bbM_{-\rho}^{\tilde{w}_0}$ is isomorphic to $\bbM_{\rho}^{\opp}$.
Therefore the central supports of $\bbM_{-\rho}$ and $\bbM_{\rho}^{\opp}$ coincide.

We refer to \cite[\S~2]{FGa} for the precise definition of $\on{Op}_{{^L}\frakg}(D_0)_{\varpi(0)}$ as a moduli scheme. Here, we describe this space in concrete terms following the style as in \S~\ref{endo alg}. Let 
$\{e,h,f\}$ be a principal $\fraks\frakl_2$-triple as with with $e\in {^L}\frakb$.  Then 
after choosing a 
uniformizer $t$ of the disc $D_0$,
$\on{Op}_{{^L}\frakg}(D_0)_{\varpi(0)}$ is the space of operators 
\begin{equation}\label{zero residue}
 \{\partial_t+\frac{f}{t}+{^L}\frakg^e\otimes \calO_0\}.
\end{equation}
Indeed, it is known from \cite[\S~9.1]{F} that the space of opers with regularities can be identified with the space of operators of the form 
$$\{\partial_t+\frac{1}{t}(f+ {^L}\frakb\otimes\calO_0)\}/{^L}U(\calO_0)\simeq  \{\partial_t+\frac{1}{t}(f+{^L}\frakg^e\otimes \calO_0)\}.$$ The condition of the residue cuts out the subspace \eqref{zero residue}. It is also explained in \emph{loc. cit.} that after a gauge transformation, \eqref{zero residue} embeds into \eqref{oper space} as a closed subscheme.

\quash{
Then 
after choosing a 
uniformizer $z$ of the disc $D_0$,
$\on{Op}_{{^L}\frakg}(D_0)_{\varpi(0)}$ is the space of operators of
the form
\[\partial_z+\frac{f}{z}+{^L}\frakb[[z]].\]
up to ${^L}U(\calO_0)$-gauge equivalence. Indeed, the space of opers with regular singularities is the space of operators of the form $\partial_z+\frac{f}{z}+\frac{1}{z}{^L}\frakb[[z]]$ up to ${^L}U(\calO_0)$-gauge equivalence. The condition that it residue is $\varpi(0)$
Equivalently, we can describe it as follows

Let us complete $f$ to an $\fraks\frakl_2$-triple
$\{e,\gamma,f\}$ with $e\in {^L}\frakn$, and $[\gamma,e]=e$. Let ${^L}\frakg^e=\oplus_{i=1}^{\ell}{^L}\frakg^e_i$ be the
centralizer of $e$ in ${^L}\frakg$, decomposed
according to the
principal grading by $\gamma$, and let $d_i$ be the corresponding degree of ${^L}\frakg^e_i$ (so $\{d_i\}$ are the exponents of ${^L}\frakg$). Then 
after choosing a 
uniformizer $z$ of the disc $D_0$,
$\on{Op}_{{^L}\frakg}(D_0)_{\varpi(0)}$ is the space of operators of
the form
\[\partial_z+f+\gamma+\sum_{i=1}^{\ell} z^{-d_i}({^L}\frakg^e_i)[[z]].\]
I.e. $\on{Op}_{{^L}\frakg}(D_0)_{\varpi(0)}$ is a torsor under the infinite dimensional affine space $\sum_{i=1}^{\ell} z^{-d_i}({^L}\frakg^e_i)[[z]]$.
}

At $\infty\in\bbP^1$, $K_\infty=\calG(0,2)(\calO_\infty)=I(2)$, so
\[\on{Vac}_\infty=\Vac_{\frakp(2)},\]
as introduced in \S~\ref{endo alg}.
As in this case $\frakp$ is the Iwahori subalgebra so $m=h$ is the Coxeter number, we also write $\on{Op}_{{^L}\frakg}(D)_{\frakp(2)}$ as defined in \S~\ref{endo alg} by
$\on{Op}_{{^L}\frakg}(D_\infty)_{1/h}$.
Choose a uniformizer $t\in F_\infty$. Recall from \eqref{oper space} and \eqref{irr oper} that
$\on{Op}_{{^L}\frakg}(D_\infty)_{1/h}\subset \on{Op}_{{^L}\frakg}(D^\times_\infty)$ is the space of operators of
the form
\[\partial_t+f+\sum_{i<\ell} t^{-d_i}V_{{^L}\frakg,d_i}\otimes \calO_\infty+t^{-d_\ell-1}V_{{^L}\frakg,d_\ell}\otimes \calO_\infty.\]
After a gauge transformation, this space is identified with
\[\{\partial_t+\frac{f}{t}+\frac{1}{t}{^L}\frakb\otimes\calO_\infty+\frac{1}{t^2}{^L}\frakg_\theta\otimes\calO_\infty\}/{^L}U(\calO_\infty),\]
where ${^L}\frakg_\theta$ is the root subspace for the highest root $\theta$ of ${^L}\frakg$ (with respect to ${^L}\frakb$).

Therefore, by Lemma \ref{support}, $\on{Op}_{{^L}\frakg}(X)_{\calG}$ is isomorphic
to
\[\on{Op}_{{^L}\frakg}(X)_{(0,\varpi(0)),(\infty,1/h)}:=\on{Op}_{{^L}\frakg}(D_\infty)_{1/h}\times_{\on{Op}_{{^L}\frakg}(D_\infty^\times)}\on{Op}_{{^L}\frakg}(\bbG_m)\times_{\on{Op}_{{^L}\frakg}(D_0^\times)}\on{Op}_{{^L}\frakg}(D_0)_{\varpi(0)}.\]
As observed in \cite[\S~5]{FG},
\begin{lem}There is a (non-canonical) isomorphism
$\on{Op}_{{^L}\frakg}(X)_{(0,\varpi(0)),(\infty,1/h)}\simeq\bbA^1$.
\end{lem}
\begin{proof}
Indeed, any ${^L}B$-torsor on $\bbG_m$ is trivial. If we fix the $\fraks\frakl_2$-triple $\{e,h,f\}$ as above,
then after choosing a coordinate $z$ on $\bbG_m$, the space of opers on $\bbG_m$ is the space of operators of the form
\[\nabla=\partial_z+ \frac{f}{z} + v , \]
where $v(z)\in {^L}\frakg^e[z,z^{-1}]$. The condition at $0$ implies that $v(z)\in {^L}\frakg^e[z]$. At the $\infty$, the local coordinate is $t=1/z$ so such an operator becomes
\[\nabla=\partial_t+\frac{f}{t} + \frac{1}{t^2}v(\frac{1}{t}).\]
The condition at $\infty$ then implies that $v\in {^L}\frakg_\theta$ is constant. Choosing a root vector $e_\theta\in {^L}\frakg_\theta$,
Then the space $\on{Op}_{{^L}\frakg}(X)_{(0,\varpi(0)),(\infty,1/h)}$ is the space of opers  of the
form
\[\nabla=\partial_z+ \frac{f}{z}+a e_{\theta},\]
with $a\in\bbC$.
So it is isomorphic to $\Spec \bbC[a]$.
\end{proof}

Now combining Proposition \ref{endo diag} and Lemma \ref{global new residue} we obtain 
\begin{lem} 
The following diagram is commutative.
\[\begin{CD}
U(V)^T@>\simeq>> \on{Fun}\on{Op}_{{^L}\frakg}(X)_{(0,\varpi(0)),(\infty,1/h)}\\
@VVV@VVh_\nabla V\\
U(V)@>>> \Gamma(\Bun_{\calG},D').
\end{CD}\]
\end{lem}

Now we can apply Corollary \ref{variant}, where $A$ is replaced by $U(V)$ in the current situation. 
Given $\chi\in \Spec U(V)$, corresponding to the character $\varphi_\chi: U(V)\to \bbC$, the point $\pi(\chi)$ on $\on{Op}_{{^L}\frakg}(X)_{(0,\varpi(0)),(\infty,1/h)}$ then can be represented by the trivial ${^L}B$-torsor on $\bbG_m$ with the connection given by
\begin{equation}\label{Kl}
\nabla=\partial_z+ \frac{f}{z}+ \varphi_\chi(g)e_\theta,
\end{equation}
where $g\in U(V)^T$ is the generator of this algebra corresponding to $a$ under the isomorphism $U(V)^T\simeq \on{Fun}\on{Op}_{{^L}\frakg}(X)_{(0,\varpi(0)),(\infty,1/h)}$, and 
\[\Aut_\calE=\omega_{\Bun_\calG}^{-1/2}\otimes (\calD'\otimes_{U(V),\varphi_\chi}\bbC)\] 
is a Hecke-eigensheaf for this connection $\calE$. Note that since $\mu^{-1}(0)$ is a Lagrangian, $\Aut_\calE$ is holonomic.

We specialize to the case when the character $\varphi_\chi:U(V)\to \bbC$ is induced by an additive character $\Psi: V\to \bbG_a$. We use $\calL_\Psi$ to denote the pullback via $\Psi$ of the exponential D-module on $\bbG_a$, which is a rank one local system on $V$. 
Since the map $U(V)\to \Gamma(\Bun_{\calG},D')$ geometrically comes from the action of $V$ on $\Bun_{\calG}$, $\Aut_\calE$ is in fact $(V,\calL_\Psi)$-equivariant on $\Bun_{\calG}$. Let $\mathring{\Bun}_{\calG}\subset\Bun_{\calG}$ be the open substack introduced before Lemma \ref{goodness}, on which $V$ acts simply-transitively. Since $\omega_{\Bun_\calG}^{-1/2}$ is canonically trivialized over this open substack (Remark \ref{sqr5}), the restriction of $\Aut_\calE$ to $\mathring{\Bun}_{\calG}$ is just $\calL_\Psi$. 

We further specialize to the case when $\Psi$ is generic (see \cite[\S~1.3]{HNY} for the meaning). Then as argued in \cite[Lemma 2.3]{HNY} (which works  in the D-module sitting without change), any $(V,\calL_\Psi)$-equivariant holonomic  D-module supported on $\Bun_\calG\setminus\mathring{\Bun}_{\calG}$ is zero (in fact, holonomicity is unnecessary for this statement). Therefore,
$\Aut_\calE$ is isomorphic to the immediate (in fact, clean) extension of $\Aut_\calE|_{\mathring{\Bun}_{\calG}}$. As the Kloosterman D-module constructed in \cite{HNY} is the eigenvalue of $\Aut_\calE$, it coincides with the connection \eqref{Kl}. On the other hand, $\Psi$ is generic if and only if the induced character $\varphi_\chi: U(V)\to\bbC$ takes non-zero value on $g$. Therefore the connection \eqref{Kl} for generic $\Psi$ also coincides with the one constructed in \cite{FG}. We are done.

\begin{rmk}Note that even $\Psi$ is non-generic, $\Aut_\calE$ is still a Hecke eigensheaf, but for a tame local system on $\bbG_m$ (since in this case $\varphi_\chi(g)=0$). However, in this case it is not easy to express $\Aut_\calE$ as an extension of the local system on $\mathring{\Bun}_\calG$ by \emph{sheaf theoretical} operations.
\end{rmk}

\end{document}